\documentclass[12pt, reqno]{amsart}
\usepackage[utf8]{inputenc}
\usepackage{amsfonts}
\usepackage{amssymb}
\usepackage{amsthm}
\usepackage{amsmath}
\usepackage{enumerate}

\usepackage[style = nature, sorting = nyt]{biblatex}

\newcommand{\mres}{\mathbin{\vrule height 1.6ex depth 0pt width
0.13ex\vrule height 0.13ex depth 0pt width 1.3ex}}

\usepackage{graphicx}

\newtheorem{thm}{Theorem}[section]
\newtheorem{cor}[thm]{Corollary}
\newtheorem{lemma}[thm]{Lemma}

\newtheorem{conjecture}{Conjecture}
\newtheorem{rmk}[thm]{Remark}
\newtheorem{defn}[thm]{Definition}
\newtheorem{prop}[thm]{Proposition}

\numberwithin{equation}{section}
\begin{document}
\title{Flipping Heegaard splittings and  minimal surfaces }
\author{Daniel Ketover}\address{Rutgers University\\  Busch Campus - Hill Center \\ 110 Freylinghausen Road, Piscataway NJ 08854 USA}
\thanks{The author was partially supported by NSF-PRF DMS-1401996.}
 \email{dk927@math.rutgers.edu}
\maketitle

\begin{abstract}
We show that the number of genus $g$ embedded minimal surfaces in $\mathbb{S}^3$ tends to infinity as $g\rightarrow\infty$.   The surfaces we construct resemble doublings of the Clifford torus with curvature blowing up along torus knots as $g\rightarrow\infty$,  and arise from a two-parameter min-max scheme in lens spaces.   More generally,  by stabilizing and flipping Heegaard foliations we produce index at most $2$ minimal surfaces with controlled topological type in arbitrary Riemannian three-manifolds.

\end{abstract}

\section{Introduction}
Heegaard splittings give natural one-parameter sweepouts of a three-manifold,  and allow one to produce index $1$ minimal surfaces from a min-max process.  The theory was developed by Simon-Smith \cite{SS} (cf. \cite{CD}, \cite{DP}) in 1983,  building on work of J.  Pitts \cite{P} and F.  Almgren \cite{Al}.    Optimal genus bounds for such constructions were obtained in \cite{K}.  In manifolds with positive Ricci curvature,  it was shown in \cite{KMN} that any lowest genus Heegaard surface may be isotoped in this way to a minimal index $1$ minimal surface.  More generally,  the author together with Y. Liokumovich and A. Song \cite{KLS} confirmed the conjecture of Pitts-Rubinstein asserting,  roughly speaking, that any strongly irreducible Heegaard surface can be isotoped to minimality (through an iterated min-max procedure) with no curvature assumption.  


A natural question is given an arbitrary three-manifold, to what extent higher parameter sweepouts of controlled topological type exist.  For instance,  it follows from Hatcher's \cite{H} proof of the Smale Conjecture that in the three-sphere there is an $\mathbb{RP}^3$ family of embedded two-spheres and an $\mathbb{RP}^2\times\mathbb{RP}^2$ family of embedded unknotted tori.  Johnson-McCullough (\cite{JM}  \cite{M}) have computed many other examples in elliptic manifolds.  One might then try to use such non-trivial higher parameter families to produce new minimal surfaces with controlled topological type.  Such an idea was exploited in \cite{HK} to construct at least two minimal two-spheres in arbitrary Riemannian three-spheres.   In another direction,  Marques-Neves \cite{MN2} introduced the Almgren-Pitts ``$p$-widths," which come from higher parameter sweepouts and give (using \cite{Z}, \cite{MN4}) minimal hypersurfaces with large Morse index but their topological type is not controlled.

We show that there are indeed natural two-parameter families of surfaces with controlled topological type in any three-manifold.    The starting point is the following foundational theorem: 

\begin{thm}[Reidemeister-Singer \cite{R}, \cite{S} (1933)]\label{rs}
If $\Sigma_1$ and $\Sigma_2$ are non-isotopic Heegaard surfaces,  then after stabilizing $\Sigma_1$ sufficiently many times the resulting surface is isotopic to a stabilization of $\Sigma_2$.  
\end{thm}

Recall that \emph{stabilizing} a genus $g$ Heegaard surface means adding a trivial unknotted $1$-handle to it,  after which it becomes a Heegaard suface of genus $g+1$.

Given Theorem \ref{rs},  if one has two non-isotopic Heegaard surfaces one can form a square of surfaces realizing the common stabilization and try to produce an index $2$ minimal surface by pulling all the surfaces tight in this square,  relative to the fixed boundary.   In fact,  D. Bachman \cite{Ba} made the following conjecture in 2002: 

\begin{conjecture}[Index $2$ minimal surfaces] \label{bach}
Let $M$ be a non-Haken $3$-manifold.  If $\Sigma_1$ and $\Sigma_2$ are strongly irreducible Heegaard surfaces in $M$ that are not isotopic to each other,  then there exists an index $2$ minimal surface with genus equal to the lowest genus stabilization of $\Sigma_1$ and $\Sigma_2$.   
\end{conjecture}

Bachman introduced a notion of ``topological" Morse index for surfaces and showed that in non-Haken manifolds, lowest genus stabilizations have topological index $2$, which he called \emph{critical surfaces}.  More generally, Bachman proposed that any critical surface could be isotoped to be geometrically minimal.

Of course not every manifold admits two distinct strongly irreducible Heegaard surfaces (for instance,  if $M=\mathbb{S}^3$) but if we take $\Sigma_2$ to be equal to $\Sigma_1$ with the opposite orientation,  we can always produce nontrivial families and obtain new min-max minimal surfaces.  

\subsection{Flipping Heegaard surfaces} 
Let $\Sigma$ be a genus $g$ Heegaard surface in an oriented $3$-manifold.   Denote $S_g(\Sigma)$ the surface obtained from stabilizing $\Sigma$ successively $g$ times.     We have the following

\begin{defn}
A Heegaard surface $\Sigma\subset M$ is \emph{flippable} if there exists an orientation-preserving isotopy of $M$ that takes $\Sigma$ to itself but with opposite orientation (equivalently,  if there exists an orientation-preserving isotopy of $M$ swapping the two handlebodies of the Heegaaard splitting determined by $\Sigma$ and fixing $\Sigma$).  
\end{defn}

Specializing Theorem \ref{rs} to this setting we have

\begin{thm}[Reidemeister-Singer (1933)]
If $\Sigma$ is a Heegaard surface,  then for $g$ large enough,  $S_g(\Sigma)$ is flippable.  
\end{thm}

Let us denote by $Flip(M,\Sigma)$ the minimal genus of a flippable surface obtained from $\Sigma$ from successive stabilizations.    There is a well-known upper bound on $Flip(M,\Sigma)$

\begin{prop}\label{maxx}
If the genus of a Heegaard surface  $\Sigma$ is $g$, then \begin{equation}Flip(M,\Sigma)\leq 2g.\end{equation}
\end{prop}

The bound from Proposition \ref{maxx} comes from considering two parallel copies of $\Sigma$ joined by a neck,  which is itself a Heegaard surface obtained from $\Sigma$ after $g$ stabilizations.   

It had been expected \cite{Ki} that in general at most one stabilization would suffice to give a flippable splitting,  but examples saturating the upper bound in Proposition \ref{maxx} were obtained by Hass-Thompson-Thurston \cite{HTT}.   

Let us now discuss the application of these topological ideas to the geometric problem of finding minimal surfaces.  Given a Heegaard surface $\Sigma\subset M$,  we can consider the set $\Pi_\Sigma$ of all Heegaard sweepouts $\{\Sigma_t\}_{t\in [-1,1]}$ which degenerate as $t\rightarrow 1$ and $t\rightarrow -1$ to the spines in the respective handlebodies (see Section \ref{prelim} for the precise statement).   We then have the following \emph{width} associated to $\Sigma$
\begin{equation}
\omega(M,\Sigma) = \inf_{\Lambda_t\in\Pi_\Sigma} \sup_{t\in [-1,1]} \mathcal{H}^2(\Lambda_t).
\end{equation}
It follows from the Isoperimetric Inequality that $\omega(M,\Sigma)>0$ and the Min-max theorem guarantees a minimal surface of index at most $1$ with total area $\omega(M,\Sigma)$.  When $M$ has positive Ricci curvature and $\Sigma$ realizes the Heegaard genus of $M$,  it follows from \cite{KMN} (Theorem 1.1) that there is an index $1$ minimal surface with area equal to $\omega(M,\Sigma)$. 


Let us a call a stationary integral varifold \begin{equation} V= n_1\Gamma_1+n_2\Gamma_2+...+n_k\Gamma_k,\end{equation} which arises from some (possibly higher-parameter) min-max procedure a  \emph{min-max minimal surface}.  The $\{n_i\}_{i=1}^k$ are positive integers and the $\{\Gamma_i\}_{i=1}^k$ are pairwise disjoint embedded minimal surfaces.  
The \emph{mass} $||V||$ of $V$ is taken to be the sum of the areas of the surfaces $\{\Gamma_i\}_{i=1}^n$ weighted according to their multiplicities.  The genus is taken to be sum of the genera of each $\{\Gamma_i\}_{i=1}^n$ weighted according to their multiplicities (see the left hand side of \eqref{genusbound}) and the \emph{Morse index} of $V$ is taken to be equal to $\sum_{i=1}^k \mbox{index}(\Gamma_k)$.  

Finally,  we need the notion of optimal foliations of a three-manifold (\cite{HK}).  

\begin{defn} An  \emph{optimal foliation by genus $g$ surfaces} is a one-parameter Heegaard sweepout $\{\Sigma_t\}_{t\in [-1,1]}$ with the further properties that
\begin{enumerate}
\item $\Sigma_0$ is an index $1$ minimal surface of genus $g$ and area $\omega(M, \Sigma_0)$,  
\item $\mathcal{H}^2(\Sigma_t) \leq \mathcal{H}^2(\Sigma_0)-Ct^2$ for all $t$ and some $C>0$
\item The family $\{\Sigma_t\}_{t\in (0,1)}$ foliates $M\setminus (\Sigma_0\cup\Sigma_1)$.  
\end{enumerate}
\end{defn}

The following is our general existence result. 

\begin{thm}[Flipping optimal  foliations]\label{flip2}
Let $(M^3,g)$ be a Riemannian three-manifold and let $\{\Sigma_t\}_{t\in [-1,1]}$ be an optimal foliation of $M$ by genus $g$ Heegaard surfaces with $\Sigma_0$ realizing $\omega(M, \Sigma_0)$.   Let 
\begin{equation}
n : = Flip(M,\Sigma_0) \mbox{ or } 2g. 
\end{equation}
Then at least one of the following holds
\begin{enumerate}
\item   $\omega(M,  S_{n-g}(\Sigma_0))< \omega(M, \Sigma_0)$,  in which case $M$ admits a min-max minimal surface of index at most $1$ with area equal to $\omega(M,  S_{n-g}(\Sigma_0))$ and genus at most $n$.
\item  $\omega(M,  S_{n-g}(\Sigma_0))= \omega(M, \Sigma_0)$ and $M$ admits infinitely many min-max minimal surfaces of genus at most $n$ and area equal to $\omega(M, \Sigma_0)$.  
\item  $\omega(M,  S_{n-g}(\Sigma_0))= \omega(M, \Sigma_0)$ and $M$ admits a min-max minimal surface $\Gamma$ so that $\Gamma\neq n\Sigma_0$ for any $n$,  $\mbox{index}(\Gamma)\leq 2$, $\mbox{genus}(\Gamma)\leq n$, $||\Gamma||>|\Sigma_0|$ and if $n=2g$ also  \begin{equation}  ||\Gamma|| < 2|\Sigma_0|.\end{equation}  
\end{enumerate}
If in addition $M$ has positive Ricci curvature,  then in case (2) there holds $n=g=Flip(M,\Sigma_0)$ and the purported infinitely many min-max minimal surfaces are connected and have genus $g$.  
\end{thm}

Theorem \eqref{flip2} implies that a manifold with an optimal Heegaard foliation always admits a second minimal surface of controlled topological type.

Note that if $M$ is endowed with a bumpy metric,  then item (2) cannot occur.   Recently Ambrozio-Marques-Neves \cite{AMN}  have found Zoll-type metrics near the round sphere for which,  in our setting,  item (2) occurs.   

Often we will apply Theorem \ref{flip2} to the case where $g$ is the Heegaard genus of $M$ and there are no minimal surfaces with area smaller than that of $\Sigma_0$.  In this situation,  item (1) also does not occur. 

The cases $n=Flip(M,\Sigma_0)$ and $n=2g$ could produce distinct minimal surfaces.  In the former case, one has better topological control but no area control, while in the latter case, one has area control but worse topological control.

In the case when $M$ is diffeomorphic to a three-sphere,  a second minimal two-sphere was obtained in \cite{HK}.  In fact, in the spherical case,  the two cases of Theorem \ref{flip2} - when $n=Flip(M, \mathbb{S}^2)$ and $n=2g =0$ \emph{coincide} as one can find a three-parameter family interpolating between the two two-parameter families considered here.   

The requirement in Theorem \ref{flip2} that one have an optimal foliation can likely be removed in the absence of stable minimal surfaces (when for instance,  $M$ has positive Ricci curvature).  Namely,  one should be able to use the mean curvature flow with surgeries to obtain an optimal foliation in any such manifold (see for instance \cite{HK} and \cite{LM})).  However,  in our applications the optimal foliation is readily available.



Considering Bachman's Conjecture \ref{bach} in the case that the surface $\Sigma_2$ is equal to $\Sigma_1$ but with the opposite orientation,  we have the following:

\begin{conjecture}\label{2}
Let $M$ be a non-Haken $3$-manifold endowed with a bumpy\footnote{Bumpiness is necessary even as one can see from the example of the round sphere.} metric.  If $\Sigma_1$ is a strongly irreducible Heegaard surface in $M$ then there exists an index $2$ minimal surface with genus equal to $Flip(M,\Sigma)$.  
\end{conjecture}

We then specialize Theorem \ref{flip2} to round lens spaces $L(p,q)$ and confirm the existence part of Conjecture \ref{2} for such manifolds.   The arguments require both the resolution of the Willmore conjecture by Marques-Neves \cite{MN} and also the resolution of the Lawson conjecture by Brendle \cite{B}  (perhaps indicating the delicacy of Conjecture \ref{2}).

The spaces $L(p,q)$ have Heegaard genus one and are flippable only in certain situations:

\begin{thm}[Bonahon-Otal \cite{BO} (1983)\footnote{Our work gives a new geometric proof of this theorem (see Remark \ref{geomnew}).}]\label{notflippable}
A genus $1$ Heegaard splitting of the lens space $L(p,q)$ is flippable if and only if $q\in\{1,p-1\}$.  
\end{thm}

We prove the following 

\begin{thm}[Genus $2$ minimal surfaces in lens spaces]\label{mininlens}
Any round lens space $L(p,q)$ with $q\notin\{1,  p-1\}$ admits a genus $2$ minimal surface $M_{p,q}$ with index at most $2$ and area satisfying \begin{equation}\label{how}
\frac{2\pi^2}{p} < |M_{p,q}| < \frac{4\pi^2}{p}.
\end{equation}
\end{thm}

It would be natural to ask whether $M_{p,q}$ has the lowest area in $L(p,q)$ above that of the Clifford torus when $q\notin\{1, p-1\}$ and also realizes the $2$-width in the volume spectrum $\{\omega_p\}_{p=1}^\infty$ of $L(p,q)$ introduced by Marques-Neves \cite{MN2} in such manifolds.  

We expect the Morse index of $M_{p,q}$ to be equal to $2$.  Since the areas of the surfaces of Theorem \ref{mininlens} are greater than that of the Clifford torus, these minimal surfaces are clearly never isoperimetric. \footnote{Viana \cite{V} solved the isoperimetric problem in lens spaces $L(p,q)$ for large $p$: the solutions are tubes around geodesics or balls.} 

We also prove a partial converse to Theorem \ref{mininlens} for the exceptional lens spaces $L(p,1)$.  
\begin{thm}
For $p$ large there exists no genus $2$ minimal surface in $L(p,1)$ with area less than $4\pi^2/p$.
\end{thm}
Let $\tilde{M}_{p,q}$ denote the lift of $M_{p,q}$ to $\mathbb{S}^3$.  We prove
\begin{thm}[Distinct lifts]\label{distinct}
For each $p$ large enough, if $L(p,q_1)$ is not isometric to $L(p,q_2)$  then $\tilde{M}_{p,q_1}\neq \tilde{M}_{p,q_2}$ (up to isometries of $\mathbb{S}^3$).
\end{thm}

Since the double cover of $L(2p,q)$ is $L(p,q)$, it follows that whenever $q$ is odd and not equal to $1$,  the lens space $L(p,q)$ contains a genus $3$ minimal surface.  Iterating we obtain
\begin{cor}
Suppose $q$ odd and not equal to one.  Then for each positive integer $n$, $L(p,q)$ admits an embedded minimal surface $A_{p,q,n}$ of genus $2^{n}+1$.  Moreover,
\begin{equation}
|A_{p,q,n}|\leq \frac{4\pi^2}{p}.
\end{equation}
\end{cor}

Some lens spaces are double covers of prism manifolds.  By Theorem 1.1 in \cite{KMN} a prism manifold admits a genus $2$ minimal surface, and thus one obtains genus $3$ minimal surfaces in some lens spaces this way.

We then study the limits of $\tilde{M}_{p,q}$ for suitable sequences $p_i\rightarrow\infty$ and $q_i < p_i$.  Expressing the three-sphere by
\begin{equation}
\mathbb{S}^3=\{(z,w)\in\mathbb{C}^2\; |\;  |z|^2+|w|^2=1\}, 
\end{equation}
we let $C$ denote the particular Clifford torus in $\mathbb{S}^3$ given by 
\begin{equation}
C= \{(z,w)\in\mathbb{S}^3\; |\; |z|^2 = \frac{1}{2}\}.
\end{equation}

\begin{thm}[Doubling of Clifford torus]\label{maincor}
For triples of positive integers $(n,m,k)$ with $n\leq m$,  $\gcd(n,m)=1$ and \begin{equation}(n,m,k)\notin\{(1,1,1), (1,2,1), (1,1,2)\}\end{equation} there exist sequences $\{p_i\}_{i=1}^\infty$ and $\{q_i\}_{i=1}^\infty$ with $\gcd(p_i,q_i)=1$, $p_i\rightarrow\infty$ and $q_i< p_i$ so that 
\begin{equation}\label{wherehappens}
\lim_{p_i\rightarrow\infty} \tilde{M}_{p_i,q_i}= 2C.
\end{equation} 
The convergence in \eqref{wherehappens} is smooth with multiplicity $2$ away from $k$ equally spaced parallel $(n,m)$-torus knots on $C$.
\end{thm} 

The case $(1,1,k)$ corresponds to curvature blowing up along $k$ parallel closed geodesics as $p_i\rightarrow\infty$.  We show using a Jacobi field argument that $(1,1,1)$ and $(1,2,1)$ cannot occur as blowup sets for minimal surfaces resembling a doubling of the Clifford torus.  The case $(1,1,2)$ is indeterminate with respect to the Jacobi field point of view but Kapouleas' \cite{Ka} gluing heuristics suggests it cannot arise.  \footnote{The family of lens spaces that could potentially give rise to a blowup set of $(1,1,2)$ are the exceptional lens spaces $L(4k, 2k\pm 1)$,  which are the only ones to admit embedded Klein bottles (cf.  Proposition \ref{lens}).}

Roughly speaking,  as $p\rightarrow\infty$, the surfaces $\tilde{M}_{p,q}$ are invariant under larger and larger groups,  and the limit is the lift of a stationary integral $1$-varifold on the two-manifold arising as a quotient from these limiting actions (possibly an orbifold).  The key to proving Theorem \ref{maincor} is to show that the stationary varifold one obtains has tangent cones of a simple type at each singular point,  which together with the area bound allow for a classification. The main ingredient is an integrated Gauss-Bonnet argument due to Ilmanen \cite{I} to study ``how much" genus may degenerate into the singular points.   

We also show for another sequence of lens spaces:
\begin{thm}[Desingularization of Lawson's Klein bottle]\label{seccor}
For any sequence of odd integers $p_i\rightarrow\infty$ 
\begin{equation}
\lim_{p_i\rightarrow\infty} \tilde{M}_{p_i,2}= \tau_{1,2}, 
\end{equation} 
where $\tau_{1,2}$ denotes the immersed Lawson Klein bottle.\footnote{This surface was introduced in \cite{L}. We describe it in detail in Section 5.1.}
\end{thm}

In their announcement, Pitts-Rubinstein \cite{PR} discussed minimal surfaces resembling the surfaces in Theorem \ref{maincor} for large $p$, but they have restrictions on their genera and were claimed to arise from non-free group actions (while these arise from free actions). 

Let $\mathcal{S}_g$ denote the space of embedded minimal surfaces of genus $g$ in $\mathbb{S}^3$ modulo isometry.  It is a long-standing question of Yau \cite{Y} whether $\mathcal{S}_g$ is finite for each $g$. Minimal surfaces in the sphere (unlike in $\mathbb{R}^3$) appear to be quite rigid and likely cannot move in continuous families.    Rigidity for certain examples was proved by Kapouleas-Wiygul \cite{KW}.

Almgren \cite{Al2} proved that $|\mathcal{S}_0|=1$ (the equator) and Brendle \cite{B} proved that $|\mathcal{S}_1| = 1$ (the Clifford torus). In a seminal paper, Lawson \cite{L} obtained the existence of a minimal surface of each genus:

\begin{thm}[Lawson (1970)]
Given any pair $m$ and $k$ of positive integers there exists an embedded minimal surface of genus $mk$ in $\mathbb{S}^3$.
In particular, for each prime number $p$, 
\begin{equation}
|\mathcal{S}_p|\geq 1, 
\end{equation}
and for each non-prime number $q$
\begin{equation}
|\mathcal{S}_q|\geq 2.
\end{equation}
\end{thm}

As their genera tend to infinity,  Lawson's surfaces converge to the union of two or more equatorial two-spheres intersecting along a great circle at equal angles.   

Since Lawson's work, several other infinite families of minimal surfaces have been discovered.  For instance, Kapouleas-Yang \cite{KY} found for each integer $n$ large enough, a  minimal surfaces of genus $n^2+1$ converging as varifolds to $2C$ in the limit that $n\rightarrow\infty$.  There are variations on this theme with the stacking of multiple Clifford tori due to Wiygul \cite{Wi} (and also along rectangular grids).  Choe-Soret \cite{CS} found families of minimal surfaces converging to a union of two orthogonal Clifford tori.  Kapouleas-McGrath \cite{KM} discovered a family resembling the doubling of the equator along parallel lines of latitude. The survey paper \cite{BS} contains a discussion of many of these results.   In each case though, not every genus is represented\footnote{To the author's knowledge, aside from Lawson's surfaces, only Kapouleas-McGrath's doublings of the equatorial two-sphere along a single geodesic and at the north and south pole represent each genus large enough.}, and for the genera that are, there are often only a bounded number of examples.  

On the other hand,  the lift of a genus $2$ minimal surface in $L(p,q)$ to $\mathbb{S}^3$ has genus $p+1$ and the number of distinct diffeomorphism types of lens spaces with fundamental group equal to $\mathbb{Z}_p$ tends to infinity as $p\rightarrow\infty$ (Lemma \ref{burnside}).  Because we can show the lifted surfaces $\tilde{M}_{p,q}$ are distinct (Theorem \ref{distinct}) we obtain definitive growth on the cardinality of $\mathcal{S}_g$:
\begin{thm}[Distinct genus $g$ minimal surfaces]
There holds
\begin{equation}\label{goto}
\lim_{g\rightarrow\infty} |\mathcal{S}_g| = \infty.  
\end{equation}
\end{thm}
In fact, the minimal surfaces giving rise to \eqref{goto} have areas below $4\pi^2$.
\subsection{Remarks}
 We will be considering Morse theory on the space of embedded surfaces of a fixed genus in a three-manifold subject to certain allowed deformations to lower genus surfaces.   An advantage of our techniques is that we do not need to compute the underlying homotopy type of this space.  The topological type of the space of embeddings of a fixed genus is known in many elliptic cases (cf Johnson-McCullough \cite{JM}).  In $\mathbb{S}^3$, for instance,  the computations follow from the statement of the Smale conjecture, proved by Hatcher \cite{H} and proved later using Ricci flow by Bamler-Kleiner \cite{BK}.   

The method is also robust with respect to finite group actions with fixed points.  To the author's knowledge,  the computation of the homotopy-type of equivariant diffeomorphisms of a three-manifold does not follow straightforwardly from Hatcher's work.  

A natural question this work poses is whether in a three-manifold,  there exist even higher parameter families of surfaces of definite genus which are non-trivial.   For instance,  for each $k$ might there be smooth families of genus $g$ surfaces that detect the Almgren-Pitts $k$ width for sufficiently large $g$?
\\
\\
The organization of this paper is as follows.  In Section 2 we introduce the terminology of min-max theory.   In Section 3 we prove the existence result Theorem \ref{flip2}.  In Section 4 we specialize to the case of lens spaces and construct the surfaces $M_{p,q}$. In Section 5 we consider the limits of the lifts $\tilde{M}_{p,q}$,  and also obtain non-existence results in the exceptional lens spaces $L(p,1)$.  In Section 6 we show that the lifts of $M_{p,q}$ to $\mathbb{S}^3$ are distinct up to isometry in non-isometric lens spaces.  In the Appendix we prove a monotonicity property for the period function associated to the Hsiang-Lawson \cite{HL} immersed tori.  

\section{Preliminaries}\label{prelim}

In this section we collect some notation and describe the Min-max existence theorem. 

Let $M$ denote a closed orientable 3-manifold and let $\mathcal{H}^2(\Sigma)$ denote the 2-dimensional Hausdorff measure of a set $\Sigma\subset M$.   

Set $I^n = [0, 1]^n \subset \mathbb{R}^n$.    Let $\{\Sigma_t\}_{t\in I^n}$ be
a family of closed subsets of $M$ and $B\subset\partial I^n$.    We call the family $\{\Sigma_t\}_{t\in I^n}$ an $n$-parameter genus g sweepout
if

\begin{enumerate}
\item $\mathcal{H}^2(\Sigma_t)$ is a continuous function of $t\in I^n$
\item  $\Sigma_t$ converges to $\Sigma_{t_0}$ in the Hausdorff topology as $t\rightarrow t_0$.  
\item For $t_0 \in I^n\setminus B$, $\Sigma_{t_0}$ is a smooth closed surface of genus g and $\Sigma_t$ varies smoothly for $t$ near $t_0$.
\item For $t\in B$,  the set $\Sigma_t$ consists of the union of a 1-complex (possibly empty) together with a smooth surface (possibly empty).
\end{enumerate}
\begin{rmk}
A Heegaard foliation, for instance,  is a sweepout $\{\Sigma_t\}_{t\in I}$ parameterized by $I=[0,1]$ where $B=\{0,1\}$ so that $\Sigma_t$ is a Heegaard surface for each $t\in (0,1)$ and $\Sigma_0$ and $\Sigma_1$ are both 1-complexes in the handlebodies determined by the Heegaard splitting.
\end{rmk}

We say that a family of subsets $\{\Sigma_t\}_{t\in \partial I^n}$ \emph{extends to a sweepout} if there
exists a sweepout $\{\Sigma_t\}_{t\in I^n}$ that restricts to $\{\Sigma_t\}_{t\in \partial I^n}$ at the boundary.   

Beginning with a genus g sweepout $\{\Sigma_t\}_{t\in I^n}$ we need to construct \\comparison sweepouts
which agree with $\{\Sigma_t\}_{t\in I^n}$ on $\partial I^n$.   We call a collection of sweepouts $\Pi$ \emph{saturated} if it satisfies
the following condition: for any map  $\Psi\in C^\infty (I^n\times M, M)$ such that for all $t\in I^n$ we have $\Psi(t,.)\in \mbox{Diff}_0(M)$ and $\Psi(t,.) = id$ if $t\in \partial I^n$, and any sweepout $\{\Lambda_t\}_{t\in I^n}\in\Pi$ we have 
 $\{\Psi(t,\Lambda_t)\}_{t\in I^n}\in \Pi$.    Given a sweepout $\{\Sigma_t\}_{t\in I^n}$, denote by $\Pi := \Pi_{\Sigma_t}$ the smallest saturated collection of sweepouts containing $\{\Sigma_t\}_{t\in I^n}$  We define the \emph{width} of $\Pi$ to be

\begin{equation}
W(\Pi,M) = \inf_{\Lambda_t\in \Pi} \sup_{t\in I^n} \mathcal{H}^2(\Lambda_t).
\end{equation}

A \emph{minimizing sequence} is a sequence of sweepouts $\{\Sigma^i_t\}\in\Pi$ such that
\begin{equation}
\lim_{i\rightarrow\infty} \sup_{t\in I^n} \mathcal{H}^2(\Sigma_t^i) = W(\Pi,M).
\end{equation}

Finally, a \emph{min-max sequence} is a sequence of surfaces $\Sigma^i_{t_i}$,  $t_i\in I^n$ taken from a minimizing
sequence so that \begin{equation}\mathcal{H}^2(\Sigma_{t_i}^i)\rightarrow W(\Pi, M). \end{equation}  

The main point of the Min-Max Theory of Almgren-Pitts (\cite{Al}, \cite{P}) as refined by Simon-Smith (\cite{SS} \cite{DP}) is that if the width is greater than the maximum of the areas of the boundary surfaces, then some min-max sequence converges to a minimal surface in $M$:

\begin{thm} [Multi-parameter Min-Max Theorem]\label{highparamminmax}
 Given a sweepout of genus g surfaces, if
\begin{equation}\label{isbigger}
W(\Pi,M)> \sup_{t\in \partial I^n} \mathcal{H}^2(\Sigma_t),
\end{equation}
then there exists a min-max sequence  $\Sigma_i := \Sigma^i_{t_i}$ such that 
\begin{equation}
\Sigma_i \rightarrow \sum_{i=1}^k n_i \Gamma_i
\end{equation} 
as varifolds where $\Gamma_i$ are smooth, closed, embedded and pairwise disjoint minimal surfaces and $n_i$ are positive integers.  Moreover, after performing finitely many compressions on $\Sigma_i$ and discarding some components, each connected component of $\Sigma_i$ is isotopic to one of the surfaces $\Gamma_i$ or to a double cover of one of the $\Gamma_i$. 

Moreover we have the following genus bounds with multiplicity:
\begin{equation}\label{genusbound}
\sum_{i\in\mathcal{O}} n_i \mbox{genus}(\Gamma_i) + \frac{1}{2}\sum_{i\in\mathcal{N}} n_i (\mbox{genus}(\Gamma_i)-1)\leq g,
\end{equation}
\noindent
where $\mathcal{O}$ denotes the subcollection of $\Gamma_i$ that are orientable and $\mathcal{N}$ denotes the subcollection of $\Gamma_i$ that are non-orientable,  and where $\mbox{genus}(\Gamma_i)$ denotes the genus of $\Gamma_i$
if it is orientable,  and the number of crosscaps that one attaches to a sphere to obtain a homeomorphic surface if $\Gamma_i$ is non-orientable.
Furthermore
\begin{equation}\label{indexbound}
\sum_{i=1}^k\mbox{index}(\Gamma_i)\leq M.
\end{equation}
\end{thm}

The index bound \eqref{indexbound} was obtained by Marques-Neves  (Theorem 1.2 in \cite{MN3}). The genus bound was obtained in \cite{K} (weaker  bounds were obtained by Simon-Smith \cite{SS},  \cite{DP}).

A fundamental question is whether one can obtain multiplicities $n_i$ greater than $1$ in the min-max theory when the metric is generic.   In the Almgren-Pitts setting,  this has recently been resolved by Zhou \cite{Z} when the ambient manifold has dimension $n$ satisfying $3\leq n\le 7$ and Chodosh-Mantoulidis \cite{CM2} when $n=3$.   Both of these works use regularizations of the area functional.  Zhou used the prescribed mean curvature functional and Chodosh-Mantoulidis used the Allen-Cahn equation.    

In the smooth setting, where one works with surfaces of a fixed topological type as in Theorem \ref{highparamminmax},  the following remains open:

\begin{conjecture}[Multiplicity One]\label{mult1}
In the setting of Theorem \ref{highparamminmax}, if the metric $g$ is assumed to be bumpy then $n_i=1$ for each $i$ such that $\Gamma_i$ is two-sided.  More generally, for any metric $g$, any two-sided $\Gamma_i$ occurring with multiplicity $n_i>1$ is stable with a non-trivial Jacobi field.
\end{conjecture} 

One quantitative tool we have in the direction of Conjecture \ref{mult1} is the Catenoid Estimate \cite{KMN} which has been useful for ruling out\\ multiplicities for unstable minimal surfaces when the multiplicity is equal to the number of parameters.  There are also \emph{ad hoc} methods to rule out stable surfaces with multiplicities developed in work with Liokumovich and Song \cite{KLS}.   


\section{Proof of Existence Result}
In this section, we prove Theorem \ref{flip2}.   We first need the following lemma:
\begin{lemma}[Controlled degeneration of stabilizations]\label{degen}
Suppose $\{\Sigma_t\}_{t\in [-1,1]}$ is an optimal genus $g$ foliation of $M$ and fix a positive integer $k$.   Then for any $\delta>0$ there exists a two-parameter sweepout $\{\Lambda_{s,t}\}$  with $t\in [-1,1]$ and $s\in [0, \varepsilon]$ so that 
\begin{enumerate}
\item For each $t\in[-1,1]$,  the surface $\Lambda_{0,t}$ is equal to the surface $\Sigma_t$ together with a union of arcs $\mathcal{A}_t$.  
\item For each $s\in (0,\varepsilon)$ and $t\notin\{0, 1\}$,  the surface $\Lambda_{s,t}$ is isotopic to $S_k(\Sigma_0)$ 
\item For each fixed $s\in (0, \varepsilon)$,  the family $\{\Lambda_{s,t}\}_{t\in [-1,1]}$ is a Heegaard sweepout\footnote{With more care one can turn this into a Heegaard foliation.} of $M$.  
\item There holds \begin{equation}\sup_{s\in (0,\varepsilon),  t\in [-1,1]} |\Lambda_{s,t}| \leq |\Sigma_0| +\delta.\end{equation}
\end{enumerate}
\end{lemma}

\begin{proof}
For each $-1<t<1$ consider $k$ unknotted arcs $\{\alpha^t_1,...,\alpha^t_k\}$ with interiors contained in $\cup_{\tau>t} \Sigma_t$ and endpoints in $\Sigma_t$.   Choose the arcs to depend smoothly on $t$ for each $-1<t<1$.   For any $\eta>0$ let $\Sigma_{t,\eta}$ denote the surface obtained from $\Sigma_t$ by adding in the boundary of the $\eta$-tubular neighborhood about $\alpha_t$ and removing the two small disks that this neighborhood makes in its intersection with $\Sigma_t$.   There exists a smooth function $f(t)>0$ so that if $\eta<f(t)$ then the surface $\Sigma_{t,\eta}$ is a piecewise smooth embedded surface of genus $g+k$ and isotopic to $S_k(\Sigma_0)$.   Moreover,  $f(t)\rightarrow 0$ as $t\rightarrow\pm 1$.   For any $s_0<1$,  consider the two-parameter family $\Lambda_{t,s} = \Sigma_{t,  sf(t)}$ (parameterized by $t\in [-1,1]$ and $s\in [0,s_0]$).   Shrinking $s_0$ enough,  and smoothing out the family, gives the desired two-parameter sweepout.

\end{proof}

It follow immediately from Lemma \ref{degen} that
\begin{cor}\label{onlyless}
If $\Sigma$ is a Heegaard surface then for each $k>0$ there holds
\begin{equation}
\omega(M, S_k(\Sigma))\leq\omega(M,  \Sigma).
\end{equation}
\end{cor}

We also need the following (see also Lemma (1.4iii) in \cite{CGK}):

\begin{lemma}[Joining Heegaard foliations (Lemma 1.6 in \cite{I})]\label{fillin}
Let $S$ and $L$ be isotopic Heegaard surfaces in $M$ and let $\{T_t\}_{t\in[-1,1]}$ be an isotopy between $S$ and $L$.  Let $\{S_t\}_{t\in[-1,1]}$ and $\{L_t\}_{t\in[-1,1]}$ be Heegaard sweepouts such that $S_{0} = S$ and $L_0 = L$.   Then there exists a 2-parameter sweep-out $\{\Sigma_{u,t}\}_{t\in [-1,1], u\in [0,1]}$ such that $\Sigma_{0,t} = S_t$,  $\Sigma_{1,t} = L_t$ and $\Sigma_{s,0} = T_s$ for any $t\in [-1,1]$ and $u\in [-1,1]$.  
\end{lemma}

Finally we need the following Lusternik-Schnirelman type result (cf.  Section 6 in \cite{MN2}).     In the following,  for any integral varifold $\mathcal{V}$, let $T_\varepsilon(\mathcal{V})$ denote the $\varepsilon$-tubular neighborhood about $\mathcal{V}$ in the $\mathcal{F}$-metric.  
\begin{lemma}\label{lus}
Let $M$ be a compact orientable Riemannian $3$-manifold.   Suppose $\{\Lambda_t\}_{t\in [0,1]}$ is a genus $n$  sweepout where 
\begin{enumerate}
\item  The smooth component $\Lambda'_0$ of $\Lambda_0$ and the smooth component $\Lambda'_1$ of $\Lambda_1$ are genus $g\leq n$ surfaces. 
\item  For each $t\in(0,1)$,  $\Lambda_t$ is isotopic to $S_{n-g}(\Lambda'_1)$ and $S_{n-g}(\Lambda'_0)$.  
\item $\Lambda'_0=\Lambda'_1$ as sets but with opposite orientation.  
\end{enumerate}
Let $\mathcal{V} = \{\Phi_1, \Phi_2, ... ,\Phi_m\}$ be a finite set of closed (possible disonnected) embedded surfaces in $M$.   Then when $\varepsilon$ is sufficiently small,  for some $t_1\in (0,1)$,  it holds that $\Lambda_{t_1}\notin T_\varepsilon(\mathcal{V})$.  
\end{lemma}

\begin{proof}
Choose a point $p\in M$ and $r>0$ so the ball $B_p(r)$ of radius $r$ about $p$ is disjoint from the support of any surface in $\mathcal{V}$ as well as $\Lambda_0$.   For any such sweepout $\{\Lambda_t\}_{t\in [0,1]}$ interchanging the handlebodies bounded by $\Lambda_0$ it holds that for some $t_0$,  the ball $B_p(r)$ has half of its volume contained in one component of each handlebody determined by $\Lambda_{t_0}$ and half in the other.  Thus by the isoperimetric inequality in $M$,  we obtain \begin{equation}\mathcal{H}^2(\Lambda_{t_0}\cap B_p(r))> \eta(M,  p,r)>0.\end{equation}

Suppose the lemma were false.  Then we have a sequence of $\varepsilon_i\rightarrow 0$ and paths of surfaces $\{\Lambda_t^i\}_{t\in [-1,1]}$ with the property that $\Lambda^i_t\in T_{\varepsilon_i}(\mathcal{V})$ for all $t\in [0,1]$.   For each $i$,  let $t_i\in[-1,1]$ be chosen according to the previous paragraph so that \begin{equation}\label{yum}\mathcal{H}^2(\Lambda_{t_i}^i\cap B_p(r))>\eta(M,  p,r)>0.\end{equation}
 Choose $\delta>0$ sufficiently small so that the (metric) tubular neighborhood
\begin{equation}
N_\delta(\mathcal{V}) :=\{x\in M\; |\; dist_M(x,\Phi_j)\leq\delta \mbox{ for some } j= 1,...,m\}
\end{equation}
is disjoint from $B_p(r)$.   it follows from the varifold convergence of $\Lambda_{t_i}^i$ to $\mathcal{V}$ that 
\begin{equation}\label{endthis}
\mathcal{H}^2((M\setminus N_\delta(\mathcal{V}))\cap \Lambda_{t_i})\rightarrow 0 \mbox{ as } i\rightarrow\infty.
\end{equation}
Since $B_p(r)\subset M\setminus N_\delta(\mathcal{V})$ it follows for large $i$, that \eqref{endthis} contradicts \eqref{yum}.
 \end{proof}

Let us now prove Theorem \ref{flip2}. 
\begin{proof}
Let $\{\Sigma_t\}_{t\in [-1,1]}$ be a optimal genus $g$ Heegaard sweepout of $M$.   Consider $S_{n-g}(\Sigma_0)$ and the corresponding width $\omega(M, S_{n-g}(\Sigma_0))$.  If \begin{equation}\omega(M, S_{n-g}(\Sigma_0))<\omega(M, \Sigma_0),\end{equation}  then applying the Min-Max theorem we obtain a minimal surface satisfying the conditions of case (1).   Thus we assume without loss of generality \begin{equation}\omega(M, S_{n-g}(\Sigma_0))\geq \omega(M, \Sigma_0). \end{equation}
By Corollary \ref{onlyless} we obtain
 \begin{equation}\label{theyareequal}\omega(M, S_{n-g}(\Sigma_0))=  \omega(M, \Sigma_0). \end{equation}
\noindent
If $n=Flip(M,\Sigma_0)<2g$,  we can form a two-parameter sweepout $\{\Gamma_{s,t}\}_{(s\in[0,1], t\in[-1,1]}$ so that for all $t\in [-1,1]$ we have
\begin{equation}
\Gamma_{0,t} = \Sigma_t \mbox{ and } \Gamma_{1,t} = \Sigma_{-t}, 
\end{equation}
and also
\begin{equation}
\Gamma_{1,t} = \Gamma_{-1,t} \mbox{ is a one-complex for all } t.
\end{equation}
Moreover, the genus of $\Gamma_{t,s}$ is equal to $Flip(M, \Sigma)$ for $(t,s)\in (0,1)\times (-1,1)$ and for each $0<s<1$,  $\{\Gamma_{s,t}\}_{t\in [-1,1]}$ is a genus $Flip(M, \Sigma)$ Heegaard sweepout of $M$.   

To accomplish this,  first invoke Lemma \ref{degen} to obtain a sweepout $\{\Gamma_{s,t}\}_{s\in [0,\epsilon], t\in[-1,1]}$ which for $s>0$ and $t\notin\{-1,1\}$ consists of genus $Flip(M,\Sigma)$ surfaces,  and which agrees (up to one dimensional set) with $\{\Sigma_t\}_{t\in [-1,1]}$ when $s=0$.  Then define for $s\in [1-\epsilon,  1]$ the surface $\Gamma_{s,t} :=-\Gamma_{1-s,t}$ (i.e., with the the opposite orientation as $\Gamma_{1-s,t}$).  This gives the desired family $\{\Gamma_{s,t}\}$ for $s\in [0,\epsilon]$ and $s\in [1-\epsilon,  1]$.   By Lemma \ref{fillin},  since there exists an isotopy from $\Gamma_{\epsilon, 0}$ to the same surface but with the opposite orientation because it is flippable,  we can fill in the family $\{\Gamma_{s,t}\}$ for $s\in [\epsilon, 1-\epsilon]$,  completing the construction of the desired two-parameter family.  

Let us now handle the case $n=2g$.   Let $T$ denote the solid closed triangle in $[-1,1]\times [-1,1]$ with boundary $\partial T$ consisting of $B=[-1,1]\times\{-1\} $,  $R=\{1\}\times [-1,1]$ and the diagonal 
\begin{equation}
D=\{(x,x)\; |\; x\in [-1,1]\}\subset [-1,1]\times[-1,1].
\end{equation}

First fix a path $L(t)_{t\in [-1,1]}$ in $M$ so that $L(t)\in\Sigma_t$  for each $t$ (for instance,  by moving normally to the Heegard foliation).  Then define the singular surfaces for $(s,t)\in [-1,1]\times [-1,1]$:
\begin{equation}
\Gamma''_{s,t} = \Sigma_s \cup \Sigma_t \cup \{L(\lambda)\; | \; \min{(s,t)}\leq\lambda\leq \max{(s,t)} \}.
\end{equation}

Note that as a varifold $\Gamma''_{t,t}$ is equal to $\Sigma_t$ with multiplicity $2$.  By the Catenoid Estimate (\cite{KMN}) we can deform $\Gamma''_{s,t}$ to obtain a family $\Gamma'_{s,t}$ parameterized by $T$,  with areas strictly less than $2|\Sigma_0|$ so that for $(t,t)\in D$,  $\Gamma'_{t,t}$ consists of a one-complex.   Moreover (up to a one-dimensional complex) we have the following equalities (up to a one-dimensonal set) on the other two boundary faces $B$ and $R$ of the solid triangle $T$:  

\begin{enumerate}
\item $\Gamma'_{s,-1} = \Sigma_s\mbox{ for all } s \mbox{ (coresponding to the ``bottom" face $B$)},$
\item $\Gamma'_{1,t} = \Sigma_{t}\mbox{ for all } t \mbox{ (corresponding to the ``right" face $R$)}.$
\end{enumerate}

We will reparameterize the triangle $T\setminus\{(1,-1)\}\subset I^2$ by new coordinates $a\in [-1,1]$ and $b\in [0,1]$.  The parameter $a$ will denote a choice of line joining a given point to $(1,-1)$, and the parameter $b$ denotes the location on this line.   This amounts to a real algebraic blowup at the point $(1,-1)$ in the parameter space.   The key point is that for each fixed choice of line $a$, by varying $b$,  we obtain a non-trivial genus $2g$ sweepout of $M$.  Let us give the details.

To that end, on $T\setminus (1,-1)$ set
\begin{equation}
b(s,t) := s- t -1, 
\end{equation} 
and
\begin{equation}
a(s,t) := \tanh(-\frac{t+1}{s-1}).
\end{equation}
Note that the quantity $\frac{t+1}{s-1}$ is the slope of the line joining $(s,t)$ to $(1,-1)$, which varies between $-\infty$ and $0$ on $T\setminus\{[1,-1]\}$, and we have used the $\tanh$ function in the definition of $a(s,t)$ simply to rescale the range of this variable. The parameter $b$ specifies which line parallel to the diagonal of the square $[-1,1\times[-1,1]$ the point $(s,t)$ lies on.  One can see easily $0\leq a(s,t)\leq 1$ and $-1\leq b(t,s)< 1$.
The inverse mappings map $[0,1]\times [-1,1)$ to $T\setminus (1,-1)$ and are given by:
\begin{equation}
t(a,b)= \frac{-1-b\tanh^{-1}(a)}{\tanh^{-1}(a)+1}, 
\end{equation}
and
\begin{equation}
s (a,b) = \frac{b+\tanh^{-1}(a)}{\tanh^{-1}(a)+1}.
\end{equation}

Let us then consider the two-parameter family $\Gamma_{a,b}$ (parameterized by $[0,1]\times [-1,1)$) given by 
\begin{equation}
\Gamma_{a,b} =  \Gamma'_{t(a,b),s(a,b)}.  
\end{equation}

\noindent

Notice that for each $a$,  as $b\rightarrow 1$, we have $(s(a,b), t(a,b))\rightarrow (1,-1)$.  Thus the family $\{\Gamma_{a,b}\}$ extends continuously to the top boundary of the rectangle in $(a,b)$-space, and thus to all of $[0,1]\times [-1,1]$.  

In both cases $n=2g$ or $n=Flip(M, \Sigma)$,  the two-parameter family $\{\Gamma_{a,b}\}_{a\in [0,1], b\in [-1,1]}$ satisfies the following properties:

\begin{enumerate}
\item  For each fixed $a\in (0,1)$, the one-parameter family $\{\Gamma_{a,b}\}_{b\in [-1,1]}$ is a genus $n$ Heegaard sweepout of $M$.  
\item The one-parameter families $\{\Gamma_{0,b}\}_{b\in [-1,1]}$ and $\{\Gamma_{1,b}\}_{b\in [-1,1]}$ are optimal genus $g$ Heegaard foliations together with a union of smooth curves (in fact, $\Gamma_{0,b} = \Sigma_b$ and $\Gamma_{1,b} = \Sigma_{-b}$ up to a one-dimensional set)
\item The orientation of $\Sigma_{0,0}$ is opposite to that of $\Sigma_{1,0}$.
\item $\sup_{a,b}|\Gamma_{a,b}| < 2|\Sigma_0|$ if $n=2g$ 
\end{enumerate}

Let $\lambda$ denote the width of the saturation $\Pi$ of this two-parameter family:
\begin{equation}
\lambda = \inf_{\{\Phi_{t,s}\}\in\Pi}\sup_{t,s} |\Phi_{t,s}|
\end{equation}

  If
\begin{equation}
\lambda> |\Sigma_0| = \omega(M,\Sigma)
\end{equation}
then the Min-max Theorem \ref{highparamminmax} applies to give a min-max minimal surface $\Gamma =\sum n_i\Gamma_i$ distinct from $\Sigma_0$.   Indeed,  if $n<2g$,  the min-max minimal surface cannot be an integer multiple of $\Sigma_0$ by the genus bounds with multiplicity \eqref{genusbound}.    If $n=2g$,  the area bound (item (4)) implies that $\Gamma$ is not an integer multiple of $\Sigma_0$.    In either case,  we have a component of $\Gamma$ distinct from $\Sigma_0$,  and thus we fall into case (3) of the theorem.  

Finally,  suppose
\begin{equation}
\lambda = |\Sigma_0|.
\end{equation}
The Min-max Theorem \ref{highparamminmax} cannot be applied in this case as \eqref{isbigger} fails (that is,  the width of the two-parameter family is not larger than the supremum of areas of its boundary values).   Let us show then that case (2) must hold.

First let us take a sequence of sweepouts $\{\Gamma^i_{a,b}\}\in\Pi$ with 
\begin{equation}\label{tight}
\max_{a,b} |\Gamma^i_{a,b}| < |\Sigma_0| + \delta_i, 
\end{equation}
where $\delta_i\rightarrow 0$.  

Let $\mathcal{S}$ denote the set of stationary integral varifolds in $M$ with mass equal to $\omega(M,\Sigma_0)$ whose genus is less than or equal to $n$ and whose support consists of pairwise disjoint embedded minimal surfaces.    For each $i>0$ and $\varepsilon>0$ let \begin{equation}\mathcal{S}^i_\varepsilon := \{(a,b)\in[0,1]\times[-1,1] \;  | \;\mathcal{F}(\Gamma^i_{a,b}.\mathcal{S})< \varepsilon\}\end{equation}  Note that for each $\varepsilon>0$ and positive integer $i$ we have $(-1,0), (1,0)\in \mathcal{S}^i_\varepsilon$.

First we claim that for each $\varepsilon>0$,  there exists an integer $I(\varepsilon)$ large enough so that if $i>I(\varepsilon)$ then $\mathcal{S}^i_\varepsilon$ contains a continuous path $(a^i_\varepsilon(\eta), b^i_\varepsilon(\eta))_{\eta\in [0,1]}\subset [0,1]\times[-1,1]$ beginning on the left side of the recangle $[0,1]\times [-1,1]$ and ending on the right side of the rectangle.   Suppose not.   Then there exists $\varepsilon_0$ so that the claim fails.    Since the claim fails,  it follows that we can find a path $(c^i_\varepsilon(\tau) ,d^i_\varepsilon(\tau))_{\tau\in [-1,1]}$ such that $(c^i_\varepsilon(0),d^i_\varepsilon(0))$ is on the bottom face of the square,  and $(c^i_\varepsilon(1),d^i_\varepsilon(1))$ is on the top face of the square so that
\begin{equation}\label{far}
\mathcal{F}(\Phi^i_{c^i_\varepsilon(\tau),d^i_\varepsilon(\tau)},\mathcal{S}^i_\varepsilon) \geq \varepsilon_0 \mbox{ for all } \tau,
\end{equation}
for some subsequence of $i$ (not relabelled).

By item (1),  we have that for each $i$,  the family $\{\Phi^i_{c^i(\tau),d^i(\tau)}\}_{\tau\in [-1,1]}$ is a genus $g$ Heegaard sweepout of $M$.  Because of \eqref{tight},  the one-parameter family $\{\Phi^i_{c^i(\tau),d^i(\tau)}\}_{\tau\in [-1,1]}$ is furthermore a minimizing sequence for genus $n$ Heegaard splittings.   On the other hand,  by \eqref{far},  it cannot be almost minimizing in annuli.   Thus by Pitts combinatorial deformation (cf.  \cite{CD}),  we obtain
\begin{equation}
\omega(M,  S_{n-g}(\Sigma_0))< \omega(M,  \Sigma_0), 
\end{equation}
a contradiction to \eqref{theyareequal}.   Thus the claim is established.  

For each $\delta>0$ there exists $\varepsilon(\delta)>0$ so that the paths $(a^i_{\varepsilon(\delta)}(\eta), b^i_{\varepsilon(\delta)}(\eta))$ joining the left side of the rectangle $[0,1]\times[-1,1]$ to the right, concatenated with the paths connecting $(a^i_{\varepsilon(\delta)}(0),b^i_{\varepsilon(\delta)}(0))$ to $(-1,0)$ and $(a^i_{\varepsilon(\delta)}(1),b^i_{\varepsilon(\delta)}(1))$ to $(1,0)$ on the left and right side, respectively,  is contained in $\mathcal{S}_\delta^i$.   This follows easily by contradiction because there is a unique point on the left side of the square (as well as on right) whose corresponding surface is minimal and has area $\omega(M,\Sigma_0)$.   Let us denote these concatenated paths by $(\tilde{a}^i_{\varepsilon(\delta)}(\eta), \tilde{b}^i_{\varepsilon(\delta)}(\eta))$.  

Let us now show that there are infinitely many embedded min-max minimal surfaces of area $|\omega(M,\Sigma_0)|$ in $M$.   Suppose toward a contradiction that there are only finitely many elements in $\mathcal{S}$.  Choosing $\delta$ small enough,  the paths $(\tilde{a}^i_{\varepsilon(\delta)}(\eta), \tilde{b}^i_{\varepsilon(\delta)}(\eta))$ (whose corresponding surfaces are contained in a $\delta$-neighborhood about $\mathcal{S}$) give a path joining $\Sigma_0$ to itself but with opposite orientation.  This violates Lemma \ref{lus}.

Let us assume now that $M$ has positive Ricci curvature.  Since $\mathcal{S}$ is compact \cite{CS2},  it follows that there exists $\delta_0$ so that whenever $\Lambda_1, \Lambda_2\in \mathcal{S}$ satisfy \begin{equation} \label{isgraph} \mathcal{F}(\Lambda_1, \Lambda_2) < \delta_0, \end{equation} then $\Lambda_2$ is a $C^\infty$ graph over $\Lambda_1$ and in particular $\Lambda_2$ is isotopic to $\Lambda_1$ through a normal exponential graphs over $\Lambda_1$.  

Fix $\delta< \frac{\delta_0}{3}$ and its corresponding $\varepsilon(\delta)$. Choose a partition $0=\tau_1<\tau_2<...< \tau_k = 1$ so that for each $j =1, 2, ..., k-1$ there exists for $i>I(\varepsilon)$ a path $(\tilde{a}^i_\varepsilon,\tilde{b}^i_\varepsilon)$ joining $(0,0)$ to $(1,0)$ contained in $\mathcal{S}^i_\delta$ and so that
\begin{equation}\label{areclose}
\mathcal{F}(\Phi^i_{\tilde{a}^i(\tau_j),\tilde{b}^i(\tau_j)}, \Phi^i_{\tilde{a}_\varepsilon^i(\tau_{j+1}),\tilde{b}_\varepsilon^i(\tau_{j+1})}) < \delta
\end{equation}

Each $\Phi^i_{\tilde{a}^i(\tau_j),\tilde{b}^i(\tau_j)}$ is by construction contained within $\mathcal{F}$-distance $\delta$ of some minimal surface $\Lambda^i_j(\delta)$ in $\mathcal{S}$.  By the triangle inequality, the consecutive minimal surfaces $\Lambda^i_j(\delta)$ and $\Lambda^i_{j+1}(\delta)$ are themselves within $\delta_0$ in the $\mathcal{F}$-metric and thus by \eqref{isgraph} can be expressed as normal graphs,  one over the next.   

For each $\delta>0$,  we have produced an ordered list of genus $g$ minimal surfaces $\{\Lambda^i_1(\delta),...,\Lambda^i_k(\delta)\}$ each within $\delta$ of the neighboring one in the $\mathcal{F}$-metric -- beginning at $\Sigma_0$ and ending at $\Sigma_0$ (but with the opposite orientation) and so that the entire one-parameter sweepout $\{\Phi^i_{\tilde{a}^i(\tau),\tilde{b}^i(\tau)}\}_{\tau\in [-1,1]}$ is contained in a $\delta$-neighborhood of these surfaces.  If there are only finitely many minimal surfaces represented among them,  this is impossible for $\delta$ small enough by Lemma \ref{lus}.

Since consecutive minimal surfaces in the list $\{\Lambda^i_1(\delta),...,\Lambda^i_k(\delta)\}$ are graphs over their neighbors, and the first and last have genus equal to $g$,  by induction we can find a smooth family of surfaces interpolating between them.  Thus we obtain that $\Sigma_0$ is flippable and $n=g$.  '
\end{proof}

\section{Minimal surfaces in lens spaces}
In round lens spaces,  we can understand exactly what new minimal-surfaces are obtained from Theorem \ref{flip2}.   
Let
\begin{equation}
\mathbb{S}^3 = \{(z,w)\in \mathbb{C}^2 \; | \; |z|^2+|w|^2=1\}.
\end{equation}
For each $p\geq 1$ and $q\geq 1$ with $q<p$ and $q$ relatively prime to $p$ we consider the cyclic $\mathbb{Z}_p$ action on $\mathbb{S}^3$ with generator $\xi_{p,q}$ \begin{equation}
\xi_{p,q} (z,w) = (e^{2\pi i /p}z,  e^{2\pi i q/p}w).
\end{equation}
We denote $L(p,q) = \mathbb{S}^3 / \mathbb{Z}_p$. 
Note that $L(p,q)$ is isometric to $L(p,r)$ when $r+q = p$ or when $qr = \pm 1\mbox{ mod } p$.  To distinguish different $q$, we will sometimes refer to this action by $\mathbb{Z}_p^q$.  

Via stereographic projection, we can consider the Hopf map \begin{equation}H:\mathbb{S}^3\rightarrow\mathbb{S}^2\end{equation}
given by 
\begin{equation}
H(z,w) = z/w\in \mathbb{C}\cup\{\infty\}. 
\end{equation}
The \emph{Hopf fibers} are the pre-images of points in $\mathbb{S}^2$ under $H$.   The $\mathbb{Z}_p$ action on $\mathbb{S}^3$ gives rise to an induced action on the Hopf fibers.  In other words,  (denoting by $[y]$ all elements in a given Hopf fiber containing $y\in\mathbb{S}^3$)  the cyclic action
\begin{equation}
\tilde{\xi}_{p,q}: \mathbb{S}^2\rightarrow \mathbb{S}^2
\end{equation}
given by the generator 
\begin{equation}
\tilde{\xi}_{p,q} [y] := [\xi_{p,q} y]
\end{equation}
is well-defined.  

The generator $\tilde{\xi}_{p,q}$ rotates points in $\mathbb{S}^2$ by angle $2\pi(q-1)/p$ about the $z$-axis and thus for $q>1$ the group generated by $\tilde{\xi}_{p,q}$ is the cyclic group $\mathbb{Z}_{k_1}$, where $k_1:=p/\mbox{gcd}(q-1,p)$.    When $q>1$, the north and south poles are fixed points of the action and thus $\mathbb{S}^2/\mathbb{Z}_{k_1}$ is an orbifold with two singular points.   This exhibits the lens space $L(p,q)$ as a Seifert fibration
\begin{equation}\label{quotient}
S: L(p,q) \rightarrow \mathbb{S}^2/\mathbb{Z}_{k_1} = \mathbb{S}^2(k_1,k_1),
\end{equation}
 where $ \mathbb{S}^2(k_1,k_1)$ denotes the orbifold with singular order $k_1$ points at the north and south poles of $\mathbb{S}^2$.  The equator in $\mathbb{S}^2(k_1,k_1)$ lifts via $S$ to a Clifford torus in $L(p,q)$.  
\subsection{Classification of minimal tori and Klein bottles} 
An oriented geodesic in $\mathbb{S}^3$ corresponds to the intersection of $\mathbb{S}^3$ with an oriented two-plane.   Therefore the space of oriented geodesics in $\mathbb{S}^3$ is homeomorphic to $\tilde{G}_2(\mathbb{R}^4)$,  the double cover of the Grassmanian $G_2(\mathbb{R}^4)$.  It is known that $\tilde{G}_2(\mathbb{R}^4)$,  is homeomorphic to $\mathbb{S}^2\times\mathbb{S}^2$.  In fact,  it will be useful to have an explicit homeomorphism.  The key for the classification is to understand how the generator $\xi_{p,q}$ acts on $\tilde{G}_2(\mathbb{R}^4)$.

For this purpose,  we may identify $\mathbb{S}^3$ with the group of unit quaternions:
\begin{equation}
\mathbb{S}^3 := \{a+bi+cj+dk,\; |\;  |a|^2+|b|^2+|c|^2+|d|^2=1\}, 
\end{equation}
and may write any quaternion as $z_0+z_1 j$, where $z_0, z_1\in\mathbb{C}$.   The inverse is given by $\overline{z}_0-z_1 j$.  We denote by $\mathbb{S}^2$ those unit quaternions with zero real value.    
There is a two-to-one map
\begin{equation}
\rho:\mathbb{S}^3\times\mathbb{S}^3\rightarrow \mbox{Isom}_+(\mathbb{S}^3)= SO(4), 
\end{equation}
given by 
\begin{equation}
\rho(x,y)(z)= xyz^{-1}.
\end{equation}
The group $\mathbb{S}^3\times\mathbb{S}^3$ acts on $\tilde{G}_2(\mathbb{R}^4)$ transitively.   Indeed, given any plane $\langle a,b\rangle\in\tilde{G}_2(\mathbb{R}^4)$ (with $a,b$ orthogonal vectors in $\mathbb{S}^3\subset\mathbb{R}^4$) we have
\begin{equation}
(q_1,q_2). \langle a,b\rangle := \langle\rho(q_1,q_2)a,  \rho(q_1,q_2)b\rangle.  
\end{equation}

The stabilizer of this action at the plane $\langle 1, i\rangle$ is
\begin{equation}
\mbox{Stab}(\langle 1,i\rangle) = \{(e^{i\theta},e^{i\phi}) \; |\; \theta,\phi \in [0,2\pi]\} = \mathbb{S}^1\times\mathbb{S}^1, 
\end{equation} 
and
\begin{equation}
\tilde{G}_2(\mathbb{R}^4)= (\mathbb{S}^3\times\mathbb{S}^3) / \mbox{Stab}(\langle 1,i\rangle)
\end{equation}
Let us define the explicit map:
\begin{equation}\label{iso}
P: \tilde{G}_2(\mathbb{R}^4)\rightarrow\mathbb{S}^2\times\mathbb{S}^2, 
\end{equation}
given by 
\begin{equation}
P((q_1,  q_2)\langle 1,i\rangle) = (q_1 i q_1^{-1},  q_2 i q_2^{-1}).
\end{equation}
The target $\mathbb{S}^2$ in \eqref{iso} denotes the unit quaternions with zero real part.

The map $P$ is well-defined and a homeomorphism (Theorem 2.7 in \cite{T}).  The points $(a,b), (-a,-b)\in\mathbb{S}^2\times\mathbb{S}^2$ correspond to the same geodesic but with opposite orientation.  The points of the form $\{\pm i\}\times p$ for $p\in\mathbb{S}^2$ correspond to the geodesics making up the fibers of the Hopf fibration $H$.   
 
For $a\in\mathbb{S}^2$,  and $B\subset \mathbb{S}^2$ let us denote 

\begin{equation}
\iota(a,B):=\{x\in\mathbb{S}^3\; |\; x\in P^{-1}(a,b)\cap\mathbb{S}^3\mbox{ for some } b\in B\}.
\end{equation}

In the same way we can define $\iota(A, b)$ for $b\in\mathbb{S}^2$ and $A\subset\mathbb{S}^2$.   If $E$ is a great circle of $\mathbb{S}^2$ and $p\in\mathbb{S}^2$ then $\iota(p,E)$ is a Clifford torus.   In light of the equivalence under antipodal reflection ($\iota(p,  E) = \iota(-p,E)$ and $\iota(p,  -E) = \iota(p,E)$) it follows that the space of Clifford tori is homeomorphic to $\mathbb{RP}^2\times\mathbb{RP}^2$.  We parameterize this space by a choice of point in the first factor, and a choice of great circle in the second factor.  Note that the set $\iota(E,p)$ is also a Clifford torus,  and in fact has the same support as $\iota(p,E)$. \footnote{This occurs because the Clifford tori are \emph{doubly-ruled} - there are two distinct mutually orthogonal families of geodesics that sweep each one out.}

We have the following:

\begin{lemma}[Section 5.1 in \cite{T})]

The isometry $\xi_{p,q}\in SO(4)$ generating $\mathbb{Z}_p^q$ corresponds to the element $$\rho(e^{\pi i (q+1)/p}, e^{\pi i (q-1)})$$ which acts on the space of oriented geodesics,  $\mathbb{S}^2\times\mathbb{S}^2$,  by rotating the first factor by angle $2\pi (q+1)/p$ about the $z$-axis and by rotating the second factor by angle $2\pi (q-1)/p$ about the $z$-axis.  
\end{lemma}

Let $k_1 = p/\mbox{gcd}(p,q-1)$ and  $k_2 = p/\mbox{gcd}(p,q+1)$.  Thus the group $\langle\xi_{p,q}\rangle =\mathbb{Z}^{q}_p$ induces cyclic action on $\mathbb{S}^2\times\mathbb{S}^2$ with order the least common multiple of $q_1$ and $q_2$. This cyclic action also extends to an action on the space of Clifford tori $\mathbb{RP}^2\times\mathbb{RP}^2$.  

Using these facts, we obtain the following classification of tori and Klein bottles in lens spaces.   Parts (5) and (6) require the proof of the Willmore conjecture \cite{MN} and resolution of the Lawson conjecture (\cite{B}).

\begin{prop}[Classification of minimal tori and Klein bottles]\label{lens}
Let $L(p,q)$ denote the lens space endowed with the round metric and $p\geq 2$.  Then the following are true:
\begin{enumerate}
\item Each $L(p,q)$ for $q\notin\{1, p-1\}$ admits exactly one Clifford torus.
\item If $p>2$ then $L(p,1)$ and $L(p,p-1)$ admit a family of Clifford tori parameterized by $\mathbb{RP}^2$.
\item $L(2,1)=\mathbb{RP}^3$ admits a family of Clifford tori parameterized by $\mathbb{RP}^2\times\mathbb{RP}^2$.   
\item  $L(p,q)$ admits an embedded Klein bottle if and only if $p=4m$ and $q=2m\pm 1$ for $m\geq 1$.    If $m>1$ then $L(4m,  2m\pm 1)$ admits an $\mathbb{S}^1$-family of minimal Klein bottles.  If $m=1$,  then $L(4,1)$ admits an $\mathbb{S}^1\times\mathbb{RP}^2$-family of minimal Klein bottles.  
\item  Any embedded minimal torus or Klein bottle in $L(p,q)$ is the projection of a Clifford torus and has area equal to $2\pi^2/p$.
\item   The least area embedded minimal surface in the lens space $L(p,q)$ for $p\neq 2$ is the projection of the Clifford torus with area $2\pi^2/p$.    In $L(2,1) = \mathbb{RP}^3$ the least area embedded minimal surface is an embedded projective plane with area $2\pi$. 
\end{enumerate}
\end{prop}

\begin{rmk}
Our computation of the space of \emph{minimal} tori coincides with the computation of the homotopy type of all genus $1$ unknotted surfaces in lens spaces due to Johnson-McCullough (Theorem 4 in \cite{JM}).    
\end{rmk}

\begin{rmk}\label{geomnew} Items (1) and (2), and (5) and (6) together with Theorem \ref{flip2} give a new geometric proof of Bonahon-Otal's result (Theorem \ref{notflippable}).  Indeed,  in $L(p,1)$,  the explicit $\mathbb{RP}^2$ family of Clifford tori exhibits the flippability of the genus $1$ splitting,  and if the other lens spaces had flippable genus $1$ splittings,  they would have to admit either infinitely many tori with area equal to $2\pi^2/p$, or a second index $1$ or $2$ minimal torus with area greater than $2\pi^2/p$ (which they do not by (5) and (6)).
\end{rmk} 

\begin{proof}
Let us consider which Clifford tori in $\mathbb{S}^3$ are invariant under the action of the group $\mathbb{Z}_p^q$. Such Clifford tori descend to minimal tori or Klein bottles in $L(p,q)$.   Recall that the element $\xi_{p,,q}$ acts on the space of geodesics $\mathbb{S}^2\times\mathbb{S}^2$ by rotating the first factor by angle $2\pi(q+1)/p$ and the second factor by angle $2\pi(q-1)/p$.  

Let us first consider the case where one of $2\pi(q-1)/p$ or $2\pi(q+1)/p$ is equal to zero modulo $2\pi$.  This happens if and only if $q=1$ or $q=p-1$.  Since $L(p,1)$ and $L(p,p-1)$ are isometric,  without loss of generality let us consider the case $q = 1$.  Then if $\iota(x, E)$ is an invariant Clifford torus,  since $2\pi(q-1)/p=0$ there is no constraint on the second factor $E$ and the point $x$ must be invariant under rotations by any multiple of angle $4\pi/p$ on the first factor.    If $p=2$,  this gives rise to no constraint in the first factor.   Thus the space of Clifford tori in $\mathbb{RP}^3$ is equal to $\mathbb{RP}^2\times\mathbb{RP}^2$.  This gives item (3). If $p=3$,  this implies the first factor $x$ is invariant under rotations by multiples of $2\pi/3$ and thus $x=\pm i$. Thus there is an $\mathbb{RP}^2$ family of such Clifford tori, $\iota(\pm i, E)$ for any equator $E$ in the lens space $L(3,1)$.   If $p=4$ this gives rise to the first factor $x$ being invariant under rotations by $\pi$.  There are two such types of invariant tori.  First we can have $\iota(\pm i,  E)$ for any equator $E$.  Secondly, we can have $\iota(\cos(\theta) j + \sin(\theta)k, E)$ for any $\theta\in[0,2\pi]$.   The first family gives an $\mathbb{RP}^2$ family of tori in $L(4,1)$,  and the second family gives a family of Klein bottles parameterized by $\mathbb{S}^1\times\mathbb{RP}^2$.  This gives the claim about $L(4,1)$ in item (4). For $p\geq 5$,  we get that $E$ can be any equator and $x$ must be invariant under a group of rotations of order at least $3$.  Thus there is only possibility for the first factor,  $x=\pm i$.  It follows that in this case the space of such Clifford tori in $L(p,1)$ is homeomorphic to $\mathbb{RP}^2$. This gives item (2).

Let us consider the case where neither $2\pi(q-1)/p$ or $2\pi(q+1)/p$ is equal to $0$ and also neither is equal to $\pi$ (modulo $2\pi$).  Then any invariant torus $\iota(x,E)$ has both $x$ and $E$ invariant under rotations of order at least $3$.  It follows that an invariant torus must be of the form $\iota(\pm i, G)$ where $G$ denotes the equator contained in the $jk$-plane.   Thus there is precisely one Clifford torus in such lens spaces and this gives rise to examples in case (1) in the Proposition.

Finally let us consider the case where neither $2\pi(q-1)/p$ or $2\pi(q+1)/p$ is equal to $0$ but (without loss of generality), $2\pi(q+1)/p =\pi$ (modulo $2\pi$).  This can happen only when $p=4k$ and $q=2k-1$.  We have already handled the case when $p=4$, so let us assume henceforth $k>1$. Any invariant torus $\iota(x,E)$ has the property that $x$ is invariant under rotations of order $2$  and $E$ is invariant under rotations of order greater than $2$.  There are again two types of invariant tori.  The first is given by $\iota(\pm i, G)$ where $G$ is the equator in the $jk$-plane.  This gives a single minimal torus, completing the proof of item (1).  The second family is of the form $K_x =\iota(x, G)$ where $x$ is an arbitrary point in the $jk$-plane.  This gives an $\mathbb{S}^1$ family of surfaces.  Let $L\subset\mathbb{S}^2$ be a circle parallel to $G$.  The group action $\xi_{p,q}$ maps $\iota(x,L)$ to $\iota(-x,L)$ which in turn is equal to $\iota(x,-L)$.   Since the group element $\xi_{p,q}$ flips the cmc tori parallel to the torus $K_x$, it follows that $K_x$ descends to a one-sided surface in $L(p,q)$, which is a Klein bottle.  This completes the proof of item (4).  

For (5), any minimal torus or Klein bottle in a lens space lifts to a connected minimal torus in $\mathbb{S}^3$ by \cite{F}. By \cite{B}, this lifted torus is a Clifford torus.

For  (6), note that any minimal surface $\Sigma\subset L(p,q)$ with area less than $2\pi^2/p$ lifts to a minimal surface $\tilde{\Sigma}$ with area less than $2\pi^2$ in $\mathbb{S}^3$.  By the resolution of the Willmore conjecture \cite{MN}, $\tilde{\Sigma}$ must be an equator,  which by Frankel's theorem \cite{F} implies that $\Sigma$ is homeomorphic to $\mathbb{RP}^2$ and $\tilde{\Sigma}$ is a double cover.  Thus the area of $\Sigma$ is $2\pi$.  Thus $p=2$ and $L(2,1)$ is diffeomorphic to $\mathbb{RP}^3$.    

\end{proof}

\subsection{Genus $2$ minimal surfaces in lens spaces}
We now prove Theorem \ref{lens}.  First we recall the following topological fact
\begin{prop}[Bonahon-Otal]\label{notf}
The lens space $L(p,q)$ is flippable if and only if $q\in\{1,p-1\}$.  
\end{prop}

Let us show
\begin{thm}[Genus $2$ minimal surfaces in lens space]
Any round lens space $L(p,q)$ with $q\notin\{1,  p-1\}$ admits a genus $2$ minimal surface $M_{p,q}$ with area satisfying \begin{equation}\label{boundsforsurface}
\frac{2\pi^2}{p} < |M_{p,q}| < \frac{4\pi^2}{p}.
\end{equation}
\end{thm}


\begin{proof}

The manifold $L(p,q)$ has an optimal foliation determined by its unique Clifford torus.   Since $L(p,q)$ is not flippable by Proposition \ref{notf},  and admits no minimal surfaces of area less than that of the Clifford torus (Proposition \ref{lens}(6)), we conclude that item (3) holds in Theorem \ref{flip2}.  Thus we obtain a minimal surface $M_{p,q}$ with area  \begin{equation}\label{areabounds}\frac{2\pi^2}{p}<|M_{p,q}|< \frac{4\pi^2}{p}.\end{equation}  From the genus bounds and classification and Proposition \ref{lens} a min-max process consisting of genus $2$ surfaces in a lens space can result in a genus $2$ minimal surface with multiplicity $1$,  a Klein bottle with multiplicity $2$ (only in the lens spaces $L(4k,  2k\pm 1)$),  or Clifford torus with multiplicity $1$.  The latter two are excluded by the area bounds \eqref{boundsforsurface}.  Thus the genus of $M_{p,q}$ is $2$.   \end{proof}

On the other hand, we have the following converse for the exceptional lens spaces $L(p,1)$ that we prove in the next section.  

\begin{prop}[Non-existence of genus 2 minimal surfaces]
For $p$ large enough, the lens space $L(p,1)$ does \emph{not} admit a genus $2$ minimal surface with area less than $4\pi^2/p$ (twice the area of the Clifford torus).   
\end{prop}

For the exceptional lens spaces $L(p,1)$ and $L(4p,2p- 1)$, we observe finally that we can obtain sharp existence for more minimal objects when the metric is not assumed to be round:
 
\begin{thm}
Let $L(p,1)$ be endowed with a metric $g$ of positive Ricci curvature.    For each $p\neq 2$,  $L(p,1)$ admits at least \emph{three} minimal tori.  
\end{thm}

\begin{proof}
By Theorem 1.1 in \cite{KMN} we obtain an index $1$ minimal torus $\Gamma_1$. Since the space of tori retracts to $\mathbb{RP}^2$ (cf.  \cite{JM}) we obtain two and three parameter family of tori (though the areas in this family are not controlled) and corresponding connected min-max limits $n_2\Gamma_2$ and $n_3\Gamma_3$.  By the genus bounds with multiplicity \eqref{genusbound},   the minimal surfaces $\Gamma_2$ and $\Gamma_3$ are tori and $n_2=n_3=1$.   If any two of the tori $\{\Gamma_1, \Gamma_2, \Gamma_3\}$ have equal areas (in particular if they coincide),  we obtain infinitely many minimal tori by Lusternick-Schnirelman theory (cf. Section 5 in \cite{MN2}).
\end{proof}

Similarly in the other exceptional lens spaces (as was  considered for minimal $\mathbb{RP}^2$ in $\mathbb{RP}^3$ in Theorem 1.7 in \cite{HK}):
\begin{thm}
Let $L(4p,2p\pm 1)$ be endowed with a metric $g$ of positive Ricci curvature.    For each $p> 1$,  $L(4p,2p\pm 1)$ admits at least \emph{two} minimal embedded Klein bottles.    
\end{thm}

\begin{proof}
Beginning with an embedded Klein bottle,  we can first minimize area using \cite{MSY} to obtain a minimal embedded Klein bottle with area $\omega_0$.   By item (4) in Proposition \ref{lens} we obtain that $L(4p,2p\pm 1)$ can be swept out by Klein bottles and we consider the corresponding one-parameter min-max problem for Klein bottles.   By the genus bounds \eqref{genusbound} we obtain a Klein bottle of area $\omega_1$ with some integer multiplicity.  The multiplicity must be odd by the genus bounds and less than $3$ by the catenoid estimate \cite{KMN} (as employed in Theorem 1.7 in \cite{HK}).  Again if $\omega_0=\omega_1$ we obtain infinitely many embedded minimal Klein bottles.
\end{proof}
\noindent

\section{Limits of minimal surfaces}
Let $\tilde{M}_{p,q}$ denote the lift of $M_{p,q}$ to $\mathbb{S}^3$.  In this section, we classify the possible limiting varifolds obtained from sequences $\tilde{M}_{p_i,q_i}$ with $p_i\rightarrow\infty$.   We will do this in stages, and first consider

\begin{thm}[Doubling of Clifford torus]\label{findlimit}
Fix $k\geq 2$. Suppose $p_i\rightarrow\infty$ with $p_i+1$ relatively prime to $k$.  Then 
\begin{enumerate}
\item If $k>2$ then in the sense of varifolds
\begin{equation}
\lim_{i\rightarrow\infty} \tilde{M}_{kp_i,p_i+1} = 2C
\end{equation}
where the convergence is smooth and graphical away from $k$ equally spaced parallel closed geodesics on $C$.  
\item If $k=2$,  the surfaces $\tilde{M}_{2p_i,p_i+1}$ converge to either $2C$ smoothly away from two equally spaced closed geodesics or else to the union of two distinct Clifford tori.  
\end{enumerate}
\end{thm} 

\begin{rmk}
From the gluing heuristics of Kapouleas \cite{Ka} we believe that the latter option occurs when $k=2$.  Such minimal surfaces were constructed by Choe-Soret \cite{CS}.
\end{rmk}

\begin{proof}


Choose $p_i\rightarrow\infty$ so that $p_i+1$ is relatively prime to $k$.    Then $L(kp_i, p_i+1)$ is a smooth manifold.  Up to taking subsequences,  $\tilde{M}_i:= \tilde{M}_{kp_i,p_i+1}$ converges to a stationary integral varifold $\tilde{V}_k$.   The support of the varifold $\tilde{V}_k$ is a union of Hopf fibers.   By the area bounds \eqref{areabounds},  the mass of $\tilde{V}_k$ is at most $4\pi^2$.   

Let $V_k$ denote the projection of $\tilde{V}_k$ to $\mathbb{S}^2$ under the Hopf map $H$.  The varifold $V_k$ is a stationary integral varifold with 
\begin{equation}\label{massbound}
||V_k||\leq  4\pi.
\end{equation} 

By \eqref{quotient} $V_k$ is also invariant under $\mathbb{Z}_k$ (the group of rotations of $\mathbb{S}^2$ about the $z$-axis by integer multiples of $2\pi/k$).    Let $G$ denote the great circle of $\mathbb{S}^2$ contained in the $xy$-plane.  We will show the following:

\begin{equation}\label{mainclaim}
V_k \mbox{ is the equator } G \mbox{ counted with multiplicity } 2 \mbox{ when } k\geq 3.
\end{equation}
\\
The density $\theta(V_k,x)$ at any point in the support $V_k$ is at most $2$.  Indeed, by considering the cone over $V_k$ in $\mathbb{R}^3$, applying the monotonicity formula and the mass bound \eqref{massbound}, we obtain for $x\in\mbox{supp}(V_k)$:
\begin{equation}\label{easier}
\theta(V_k,x)\leq \frac{||V_k||}{2\pi}\leq 2.
\end{equation}

By \eqref{easier} the tangent cone of the varifold $V_k$ at a singular point is a triple junction (with multiplicity $1$) or union of two lines.  But triple junctions with multiplicity $1$ are impossible (as orientable closed surfaces cannot have such a limit) and therefore the tangent cone at singular points of $V_k$ consists of two lines.  Thus the support of $V_k$ is a union of immersions, and consists of two great circles (by the mass bound \eqref{massbound}).  For $k\geq 5$ and $k=3$ the only $\mathbb{Z}_k$-invariant such configuration is $V_k=2G$ (an additional argument is needed for $k=4$).

In the following, we will give a rather more elaborate argument (independent of the density bound \eqref{easier}) to prove Claim \ref{mainclaim} because we need to apply the argument to classifying stationary integral varifolds on orbifold two-manifolds where the above argument \eqref{easier} does not apply (Theorem \ref{other}).  

To that end, suppose the claim \eqref{mainclaim} were false.  If the support of $V_k$ were smooth,  it must be equal to $G$ since $G$ is the only closed geodesic invariant under the group of rotations of order greater or equal to $3$ about the $z$-axis.  If $V_k$ is equal to $G$ with some multiplicity $n$, by Allard's theorem \cite{All} (as the genus of $\tilde{M}_i$ is equal to $kp_i+1$) it follows that $n\neq 1$.  By the mass bound of $4\pi$ we get that $n\leq 2$ so that $V_k = 2G$.   Thus to prove \eqref{mainclaim} it remains to rule out that $V_k$ is a non-smooth stationary integral varifold with some non-empty singular set $\mathcal{S}\subset\mathbb{S}^2$.

Let us partition the singular set \begin{equation}\mathcal{S}=\mathcal{S}_0\cup\mathcal{S}_1\end{equation} where $\mathcal{S}_0:= \mathcal{S}\cap\{N,S\}$ (where $N$ and $S$ denote the north and south pole of $\mathbb{S}^2$, respectively) and $\mathcal{S}_1=\mathcal{S}\setminus\mathcal{S}_0$.  Note by the $\mathbb{Z}_k$-equivariance,  the cardinality of $\mathcal{S}_1$ is a multiple of $k$. 

For any open set $\mathcal{O}\subset\mathbb{S}^2$ with $\mathcal{O}\cap\mbox{supp}(V_k)\neq \emptyset$, let us say that \emph{$\tilde{M}_i\rightarrow V_k$ smoothly in $H^{-1}(\mathcal{O})$} if for any compact $K\subset \mathcal{O}$, the surface $\tilde{M}_i\cap H^{-1}(K)$ can be written as a union of exponential graphs over $H^{-1}(V_k\cap K)$ that each converges smoothly to $H^{-1}(V_k\cap K)$ as $i\rightarrow\infty$.  Let us say \emph{$\tilde{M}_i\rightarrow V_k$ non-smoothly in $H^{-1}(\mathcal{O})$} if for some compact $K\subset\mathcal{O}$ the previous statement fails.  

Let us define the extended singular set:

\begin{equation}
\mathcal{D} := \mathcal{S}\; \cup \; \{y\in\mbox{reg}(V_k)\footnote{$\mbox{reg}(V_k)$ denotes the supoort of the regular part of the stationary integral varifold $V_k$.}\;|\;\mbox{for all } r>0,\;  \tilde{M}_i \rightarrow \tilde{V}_k \mbox{ non-smoothly in } H^{-1}(B_r(y)) \}
\end{equation}
\noindent
Note that $\mathcal{S}\subset\mathcal{D}$ but $\mathcal{D}$ might be a strictly larger set \footnote{If there is convergence to $\tilde{V}_k$ with multiplicity greater than $1$,  for instance,  the set $\mathcal{D}$ includes the projection to $\mathbb{S}^2$ of the locations along which necks are collapsing.}.

First we show the following: 

\begin{equation}\label{claim4}
\mbox{If } \mathcal{D}\setminus\mathcal{S}_0\neq\emptyset\mbox{, then } |\mathcal{D}|=|\mathcal{S}_1| =k.
\end{equation}

Toward proving \eqref{claim4} first fix $x\in\mathcal{D}\setminus\mathcal{S}_0$ and consider a closed ball $\overline{B}_x(r)$ about $x$ with radius $r$ so small so that $V_k\cap \overline{B}_r(x)$ consists of geodesic segments emanating from $x$ (possibly some with multiplicities) which meet $\partial B_x(r)$.  Consider the solid torus in $\mathbb{S}^3$\begin{equation}\tilde{T}_x := H^{-1}(\overline{B}_r(x))\end{equation} and its boundary \begin{equation}\partial \tilde{T}_x := H^{-1}(\partial B_r(x))\end{equation} as well as the lifted surface with boundary \begin{equation}\tilde{S}'(x,i) := \tilde{M}_{i}\cap \tilde{T}_x \end{equation} where by perhaps slightly perturbing $r$ we have for each $i$: \begin{equation}\partial \tilde{S}'(x, i) := \tilde{M}_{i}\cap \partial\tilde{T}_x. \end{equation}   

By the monotonicity formula in $\mathbb{S}^3$,  there exists $\varepsilon_0>0$ so that any component of a minimal surface (with boundary) contained in $\tilde{T}_x$ that also intersects $H^{-1}(B_{r/2}(x))$ has area at least $\varepsilon_0 r^2$.   Denote by $\tilde{S}(x,i)$ the union of those components of $\tilde{S}'(x,i)$ with area at least $\varepsilon_0 r^2$.  By the monotonicity formula and area bounds,  there are only finitely many elements in $\tilde{S}(x,i)$ and we may pass to a subsequence so that the number of elements is constant.  Note that in the sense of varifolds 
\begin{equation}
\lim_{i\rightarrow\infty} \tilde{S}(x,i) = \tilde{V}_k\mres H^{-1}(B_{r/2}(x)).
\end{equation}


The subgroup $\mathbb{Z}_{p_i}=\langle \xi^k_{kp_i, p_i+1}\rangle $ of $\mathbb{Z}_{kp_i}$ generated by $\xi^k_{kp_i, p_i+1}$ acts freely on the manifold with boundary $\tilde{T}_x$ and the quotient space $T_x=\tilde{T}_x/\mathbb{Z}_{p_i}$ is a solid torus diffeomorphic to a subset of $L(kp_i,p_i+1)$.  Set $S(x,i):= \tilde{S}(x,i)/\mathbb{Z}_{p_i}$ and $\partial S(x,i) := \partial \tilde{S}(x,i)/\mathbb{Z}_{p_i}$.  

Let us denote by $W(x,i)$ the number of components of $\partial S(x,i)$,  by $W_1(x,i)$ the number among these that are homologically non-trivial in $T_x$ and by $W_0(x,i)$ the number that are homologically trivial.  Similarly, let $\tilde{W}(x,i)$, $\tilde{W}_1(x,i)$ and $\tilde{W}_0(x,i)$ denote (respectively) the total number of components of $\partial \tilde{S}(x,i)$,  homologically non-trivial components, and homologically trivial components in $\tilde{T}_x$.

Note that $p_i W_0(x,i) =\tilde{W}_0(x,i)$.   We claim that for large $i$, \begin{equation}\label{gt}\tilde{W}_1(x,i)=W_1(x,i).\end{equation}
 
To see \eqref{gt}, let us denote by $(n,m)$ the isotopy class of an embedded curve on $\partial T_x$ that goes around the torus $m$ times meridianally and $n$ times longitudinally.   By definition any $(n,m)$ curve counted in the sum $W_1(x,i)$ has $n\neq 0$ but for large $i$ such a curve also has $m=0$.   Otherwise lifting the curve to $\tilde{T}_x$ it would cross every Hopf fiber on $\partial \tilde{T}_x$ at least $p_i$ times,  and thus force $V_k$ to contain a small circle,  which is clearly not stationary.  Thus $m=0$.  But the only embedded curves $(n,0)$ on $\partial T_x$ occur when $n=\pm 1$.   Thus all curves counted in the sum $W_1(x,i)$ are $(\pm 1, 0)$ curves, and lift to connected curves (isotopic to a Hopf fiber on $\partial\tilde{T}_x$). This gives that $\tilde{W}_1(x,i)=W_1(x,i)$.

Furthermore,  the integer $W_1(x,i)$ is even.  Indeed, observe that the curves contributing to the sum $W_1(x,i)$ are parallel copies of the same curve $(\pm1, 0)$ which add up to the trivial element in  $H_1(T_x; \mathbb{Z})$ (since the curves counted in $W_1(x,i)$, together with trivial curves bound an orientable surface).   In order for the sum of these  curves to be equal to zero in $H_1(T_x; \mathbb{Z})$,  the number $W_1(x,i)$ must be even.  


For large $i$, any connected component of $S(x,i)$ lifts to a connected component of $\tilde{S}(x,i)$ (otherwise the component lifts to $p_i$ elements in $S(x,i)$, which violates the uniform bound on the cardinality of $S(x,i)$ for large $i$).  Then by the multiplicativity of the Euler characteristic under covering maps,  it follows that the genus of $\tilde{S}(x,i)$ is given by 
\begin{align}\label{lower}
\mbox{genus}(\tilde{S}(x,i))&=  |S(x,i)|(1-p_i)+p_i\mbox{genus}(S(x,i)) +\frac{p_iW(x,i)}{2}-\frac{\tilde{W}(x,i)}{2} \\
&  =|S(x,i)|(1-p_i)+p_i\mbox{genus}(S(x,i)) +(p_i-1) \frac{W_1(x,i)}{2}. \\
& =(p_i-1) (\frac{W_1(x,i)}{2}-|S(x,i)|) + p_i\mbox{genus}(S(x,i))\label{eee} \\&
\geq -1 + p_i\label{ee}\mbox{     for large $i$}.
\end{align}
The second equality follows from the fact that $W_0(x,i)=p_i \tilde{W}_0(x,i)$ and $W_1(x,i)=\tilde{W}_1(x,i)$.  


Let us now justify the inequality \eqref{ee}.  There are three cases,  depending on whether $\mbox{genus}(S(x,i))$ is equal to $0$, $1$ or $2$.  Since $T_x$ is homeomorphic to a subset of the lens space, and $M_{i}$ has genus $2$, these are the only possible cases.

First let us assume $\mbox{genus}(S(x,i))=0$.  Then we claim that for $i$ large enough: \begin{equation}\label{comeon} W_1(x,i)/2 - |S(x,i)| \geq 1.\end{equation}  Otherwise $W_1(x,i)/2 - |S(x,i)| = 0 $ and by \eqref{eee} the genus of $\tilde{S}(x,i)$ would be uniformly bounded for large $i$.   Since the areas of $\tilde{S}(x,i)$ are uniformly bounded independent of $i$,  Ilmanen's integrated Gauss-Bonnet argument (\cite{I},\cite{K}) gives that
\begin{equation}\label{second}
\int_{\tilde{S}(x,i)\cap H^{-1}(B_{r/2}(x))}|A|^2 < \Lambda, 
\end{equation}
for some $\Lambda$ independent of $i$. But \eqref{second} implies that the convergence of $\tilde{S}(x,i)$ to $\tilde{V}_k$ is smooth in $H^{-1}(B_{r/2}(x))$ away from finitely many points in the interior of $H^{-1}(B_{r/2}(x))$.  As the lifted surfaces $\tilde{S}(x,i)$ are becoming more and more periodic as $i$ increases, this implies the convergence of $\tilde{S}(x,i)$ to $\tilde{V}_k$
is smooth and graphical on compact subsets of $H^{-1}(B_{r/2}(x)))$.   But this violates the assumption that $x\in\mathcal{D}$.  Thus we obtain \eqref{comeon} and the bound $\mbox{genus}(\tilde{S}(x,i))\geq p_i-1$ in \eqref{ee}.

If $\mbox{genus}(S(x,i))=1$, then because $(p_i-1) (W_1(x,i)/2-|S(x,i)|)$ is an integer, we obtain that \begin{equation}\label{b}W_1(x,i)/2-|S(x,i)|\geq 0, \end{equation} and thus $\mbox{genus}(\tilde{S}(x,i))\geq p_i$ (otherwise, the same argument as in the case $\mbox{genus}(S(x,i))=0$ gives a contradiction to the fact that $x\in\mathcal{D}$). Finally, for the same reason if $\mbox{genus}(S(x,i))=2$, then we obtain from \eqref{eee} that for large $i$
\begin{equation}\label{c}W_1(x,i)/2-|S(x,,i)|\geq -1.
\end{equation}
Thus $\mbox{genus}(\tilde{S}(x,i))\geq p_i+1$. in this case as well.  This completes the proof of inequality \eqref{ee}.

Let $\mathcal{D}(x)$ denote those elements of $\mathcal{D}$ in the orbit of $x$ under the group $\mathbb{Z}_k$ of rotations by integral multiples of $2\pi/k$ about the $z$-axis.  Let \begin{equation}\tilde{R}(i) = \tilde{M}_{i}\cap (\mathbb{S}^3\setminus\bigcup_{y\in\mathcal{D}(x)} \tilde{T}_y).\end{equation}
Note that 
\begin{equation}
kp_i +1  =  \mbox{genus}(\tilde{M}_{i}) \geq\mbox{genus}(\tilde{R}(i))+ \sum_{y\in\mathcal{D}(x)} \mbox{genus}(\tilde{S}(y,i)). 
\end{equation}
Thus we obtain from \eqref{ee}
\begin{equation}
kp_i +1 \geq kp_i+\mbox{genus}(\tilde{R}(i)) - k, 
\end{equation}
which gives
\begin{equation}\label{totalbounded}
\mbox{genus}(\tilde{R}(i)) < k+1.
\end{equation}
By \eqref{totalbounded} and Ilmanen's integrated Gauss-Bonnet argument we again obtain that for any compact set composed of Hopf fibers $\tilde{K}\subset \mathbb{S}^3\setminus \cup_{y\in\mathcal{D}(y)}T_y$ we have 
\begin{equation}
\int_{\tilde{K}\cap\tilde{M}_i} |A|^2 \leq\Lambda, 
\end{equation}
for some $\Lambda$ independent of $i$. Because $\tilde{M}_i$ are becoming more and more periodic, this implies that $\tilde{M}_i$ converge smoothly to $\tilde{V}_k$ in $\tilde{K}$.  Setting $K:=H(\tilde{K})\subset\mathbb{S}^2$ we thus obtain
\begin{equation}
K\cap\mathcal{D} = \emptyset.
\end{equation}
\noindent
By repeating the above argument with a sequence of $r_i$ tending to zero, we obtain \begin{equation}\mathcal{D}\setminus\mathcal{D}(x) =\emptyset.\end{equation}  
\noindent
This establishes claim \eqref{claim4}.   Note that we also must have \emph{equalities} in each of the three cases \eqref{comeon}, \eqref{b} and \eqref{c}.   Otherwise,  by \eqref{eee} the surfaces $\tilde{S}(x,i)$ have too much genus.  

Let us continue the proof of \eqref{mainclaim}.  Suppose $x\in\mathcal{D}\setminus\mathcal{S}_0$.  Then by \eqref{claim4} we get $x\in\mathcal{S}_1$.  Since \eqref{claim4} implies that $\mathcal{D}$ contains only the iterates $\mathcal{D}(x)$ it now follows \emph{a posteriori} that the convergence
\begin{equation}
    \tilde{M}_i\rightarrow \tilde{V}_k \mbox { in }H^{-1}(B_{2r}(x)\setminus B_{r/2}(x))
\end{equation}
is smooth.  Thus $W_1(x,i)\geq 4$ as it is an even integer greater than $2$ since $x\in\mathcal{S}_1$.  Also, $W_0(x,i)= 0$ for each $i$.

We next claim:
\\
\\
 \emph{If $x\in\mathcal{S}_1$ then $W_1(x,i) = 4$ for large $i$,  and the tangent cone of the stationary integral varifold $V_k$ at $x$ consists of a union of two distinct multiplicity $1$ lines.}
\\

To prove the claim we consider the three cases $\mbox{genus}(S(x,i))=0$,\\ $\mbox{genus}(S(x,i))=1$, and $\mbox{genus}(S(x,i))=2$.  If $\mbox{genus}(S(x,i))=0$ then we must have $W_1(x,i)=4$ and $|S(x,i)|=1$ for large $i$. To see this, note that by the genus bound and \eqref{eee} we obtain,
\begin{equation}\label{endplease}
\frac{W_1(x,i)}{2}-|S(x,i)| = 1.
\end{equation}
Assume $W_1(x,i)>4$.  Then \eqref{endplease} implies by the pigeonhole principle that at least one component $\tilde{C}(x,i)\in \tilde{S}(x,i)$ has exactly two boundary components among the curves counted in $W_1(x,i)$.  The component $\tilde{C}(x,i)$ converges to a stationary integral varifold $\tilde{B}_k$ supported on $\tilde{V}_k$.  The convergence of this component $\tilde{C}(x,i)$ must be smooth (though perhaps with multiplicity) in the solid torus $H^{-1}(B_{r/4}(x))$ since it has zero genus by applying \eqref{eee} to the component $\tilde{C}(x,i)$ (in place of $S(x,i)$),  and repeating the foregoing argument.   By this smooth convergence,  the other components in $\tilde{S}(x,i)\setminus\tilde{C}(x,i)$ lie on one side of $\tilde{C}(x,i)$ in $H^{-1}(B_{r/4}(x))$ for large $i$.  Therefore the tangent cone to $V_k$ at $x$ is contained in a  half-space and is a line with some multiplicity.  This contradicts the fact that $x\in\mathcal{S}_1$.   Thus we have established $W_1(x,i) = 4$ if $\mbox{genus}(S(x,i))=0$.

If $\mbox{genus}(S(x,i))=1$, then we must have 
\begin{equation}
\frac{W_1(x,i)}{2}-|S(x,i)|= 0.
\end{equation}
Thus each element in $S(x,i)$ has exactly two boundary curves counted by $W_1(x,i)$.  Since $W_1(x,i)\geq 2$, one of these connected components has genus $1$, and the other(s) have genus $0$.  Applying \eqref{eee} to the lift of this genus $0$ component, we get this lifted component of $\tilde{S}(x,i)$ also has bounded genus and so converges smoothly to $\tilde{V}_k$ in $H^{-1}(B_{r/4}(x))$.  As in the previous case, this implies that the tangent cone of $V_k$ at $x$ consists of a line with multiplicity which contradicts the fact that $x\in\mathcal{S}_1$.  Thus the case $\mbox{genus}(S(x,i))=1$ is impossible.  

The case $\mbox{genus}(S(x,i))=2$ is impossible as well.  As remarked above, we must have 
\begin{equation}
\frac{W_1(x,i)}{2}-|S(x,i)|= -1, 
\end{equation}
which is impossible since each component of $S(x,i)$ has at least two boundary components counted in $W_1(x,i)$ (recalling that $W_0(x,i) = 0$).  

Thus we obtain $W_1(x,i)=4$ and $|S(x,i)| = 1$ and $\mbox{genus}(S(x,i)) = 0$ for large $i$ as claimed. \footnote{This is exactly the situation one obtains with a fundamental domain of the singly-periodic Scherk surface in $\mathbb{R}^3$ as they degenerate to two planes,  which is what the degeneration at $H^{-1}(x)$ is resembling.}

Since $W_1(x,i)=4$ for large $i$,  it follows that the sum of the multiplicities of geodesic segments of $V_k$ meeting at $x$ is equal to $4$,  and thus the tangent cone at any point in $\mathcal{S}_1$ is a union of two distinct lines or one line with multiplicity $2$.   It cannot be a multiplicity $2$ line because $x$ is a singular point of the varifold $V_k$ (i.e.,  since $x\in\mathcal{S}_1\subset\mathcal{S}$).   Thus the tangent cone is a union of two distinct lines at $x$ as claimed.

Any stationary integral varifold with singular points consisting of such tangent cones is a union of immersed closed geodesics.  Therefore if $\mathcal{S}_1$ is non-empty, the varifold $V_k$ consists of a union of great circles on the $2$-sphere,  each with multiplicity $1$ and furthermore (by \eqref{claim4}) $\mathcal{S}_0$ is empty. If $x\in\mathcal{S}_1$ is in the northern or southern hemisphere of $\mathbb{S}^2$ and $k\geq 3$, then $\mathcal{S}_1$ consists of the $k$ iterates of the point $x$.  Since a great circle joining two such points in the northern hemisphere can intersect none of the other iterates, and at least two great circles must meet at every iterate for the iterate to be a singular point of $V_k$, such a configuration includes in its support at least $k$ great circles.  But for $k\geq 3$, this has mass at least $6\pi$, which is larger than $4\pi$ and thus impossible (moreover, such a configuration would  contain necessarily more than $k$ singular points, in violation of \eqref{claim4}).  Thus $x$ cannot be in the northern or southern hemisphere if $k\geq 3$.  


If $x$ is contained on the equator $G$, then for $k\geq 3$, the only possible configuration is when $G$ it attained with multiplicity $1$ which violates the fact that $x$ was assumed to be in $\mathcal{S}_1$.   Since all cases lead to contradictions,  we obtain $\mathcal{S}_1=\emptyset$.


We are left to rule out the case where $\mathcal{S}_1$ is empty and $\mathcal{S}_0$ contains either one or two points. Certainly there is no stationary integral varifold on $\mathbb{S}^2$ with one singular point, and so we must rule out that the singular set of $V_k$ consists of $N$ and $S$, the north and south poles of $\mathbb{S}^2$.   Since $V_k$ contains no other singular points it consists of a number of half circles joining the north pole to the south pole (counted with multiplicities). The least mass such a stationary varifold consists of $k$ half-circles with total length equal to $k\pi$.   We obtain from the area bound, 
\begin{equation}
k\pi\leq ||V_k|| \leq 4\pi, 
\end{equation}
so that $k\leq 4$.   

If $k=3$, $V_k$ consists of three equally spaced half circles.  But such a configuration is not the limit of any closed embedded orientable surfaces in the sphere.  For $k=4$,  we must rule out that $V_k$ consists of two perpendicular closed geodesics that intersect at $N$ and $S$.  As before, let $\tilde{T}_N$ denote the lift under $H$ of a small ball $B_r(N)$ about $N$.  Note that $T_N := \tilde{T}_N/\mathbb{Z}_{4p_i}$ is homeomorphic to a subset of $L(4p_i, p_i+ 1)$ and if we set as before \begin{equation}\tilde{S}(N,i)=\tilde{M}_{i}\cap T'_N\end{equation} then we obtain that $\partial S(N,i):=\partial \tilde{S}(N,i)/\mathbb{Z}_{4p_i}$ contains a \emph{single} non-trivial closed $(4,1)$ curve.  This is due to the Seifert fibered structure of the lens space as specified in \eqref{quotient} and because the convergence of $\tilde{M}_i$ to the union of Clifford tori making up $\tilde{V}_k$ is smooth away from $N$ and $S$ by claim \eqref{claim4} as $\mathcal{D}\setminus\mathcal{S}_0=\emptyset$.

The homology group $H_1(T_N; \mathbb{Z})$ of the solid torus $T_N$ is isomorphic to $\mathbb{Z}$ with generator $[\alpha]$, a $(1,0)$ curve.   The curve $\partial S(N,i)$ is equal to $4[\alpha]\neq[0]\in H_1(T_N;\mathbb{Z})$.  Thus there is no orientable surface contained in $T_N$ with boundary equal to $\partial S(N,i)$.  This is a contradiction since $\tilde{M}_{i}$ is orientable.   (In fact $\partial S(N,i)$ is trivial in $H_1(T_N;\mathbb{Z}_2)=\mathbb{Z}_2$ and there is a non-orientable surface with boundary $\partial S(N,i)$).   
Thus we have ruled out the case $|\mathcal{S}_0|=2$ and \eqref{mainclaim} is established. 

It remains to show that $\tilde{V}_{k}\rightarrow 2C$ smoothly away from $k$ parallel great circles.  If $\mathcal{D}$ is empty, then the convergence of $M_i$ to $2C$ would be smooth, which is impossible.  By claim \eqref{claim4}, $\mathcal{D}$ must contain exactly $k$ points obtained by iterating $x$ under the group of rotations $\mathbb{Z}_k$.  This implies that  \begin{equation}
\lim_{i\rightarrow\infty} \tilde{M}_{i} = 2C
\end{equation}
and that the convergence is smooth away from $k$ parallel equally spaced closed geodesics.  This completes the proof when $k\geq 3$.

Finally we remark on the case $k=2$.  The proof is the same up the point that we classified stationary varifolds $V_k$ in the case that $\mathcal{S}_0$ is non-empty.   Here we might obtain two great circles intersecting at $N$ and $S$.  The contradiction we reached in the case $k=4$ ruling out such a configuration (by studying the surfaces in $S(N,i)$) does not apply, as one obtains that $\partial S(N,i)$ consists of two $(2,1)$ curves, which may in fact bound an orientable surface in $S(N,i)$.  

\end{proof}


\begin{rmk}
An alternative approach to Theorem \ref{findlimit} is to use the fact that  $\mbox{index}(M_i)\leq 2$ to argue that the cardinality of $\mathcal{S}_1$ is at most $2k$.  Since one still has then to reduce this bound to $k$ and thereby consider the accumulation of genus, we have chosen not to argue this way.   
\end{rmk}

\subsection{Lawson immersed tori}\label{lawsonsection}
To understand the remaining limits of $\tilde{M}_{p_i,q_i}$ as $p_i\rightarrow\infty$ we first need to recall the minimal immersed tori and Klein bottles $\tau_{n,m}\subset\mathbb{S}^3$ obtained by Lawson (Theorem 3 in \cite{L}, see also Hsiang-Lawson \cite{HL} and Penskoi \cite{P}).  

For positive and relatively prime $n$ and $m$ integers consider the immersion \begin{equation}\Phi_{n,m}:\mathbb{R}^2\rightarrow\mathbb{S}^3\subset\mathbb{R}^4
\end{equation}
given by 
\begin{equation}
    \Phi_{n,m}(x,y) = (\cos(nx)\cos(y),\sin(nx) \cos(y),\cos(mx)\sin(y),\sin(mx)\sin(y))
    \end{equation}
Restricting $\Phi_{n,m}$ to a fundamental domain gives an immersed minimal torus when neither $n$ nor $m$ is even, and a minimal Klein bottle when one of them is even.  While the embedding is manifestly invariant under the lattice generated by $(0,2\pi)$ and $(2\pi,0)$ it is in fact generated by a smaller lattice whose fundamental domain has half the area. Note also that $\tau_{n,m}$ is isometric to $\tau_{m,n}$ and so we will assume $m>n$.  The surface $\tau_{1,1}$ is the Clifford torus $C$.

The induced metric is on $\tau_{n,m}$
\begin{equation}
ds^2= (n^2 \cos^2(y) + m^2\sin^2(y))dx^2 + dy^2
\end{equation}
with volume element
\begin{equation}
dvol = \sqrt{n^2 \cos^2(y) + m^2\sin^2(y)}dxdy.
\end{equation}
If both $n$ and $m$ are greater than $1$, we obtain
\begin{equation}\label{4bigger}
Area(\tau_{n,m})\geq 4\pi^2 = 2|C|.
\end{equation}
If $m\geq 1$ and $n=1$, on the other hand, we may express the area in terms of elliptic integrals:
\begin{equation}\label{elliptic}
Area(\tau_{1,m}) = \pi\int_0^{2\pi}\sqrt{1+ (m^2-1)\sin^2(y)}dy.
\end{equation}
The set of numbers $\{\tau_{1,m}\}_{m=1}^\infty$ is a monotone increasing sequence.  In fact 
\begin{equation}\label{2less} Area(\tau_{1,2})\approx 30.44 < 4\pi^2 \approx 39.48\end{equation}  but 
\begin{equation} Area(\tau_{1,3})\approx 41.99  > 4\pi^2.\end{equation}
Therefore for $m>2$, it holds that
\begin{equation}\label{3bigger}
Area(\tau_{1,m}) > 4\pi^2 = 2|C|.  
\end{equation}
We will show in fact that for certain sequences $\{p_i\}_{i=1}^\infty$ and $\{q_i\}_{i=1}^\infty$ the surfaces $\tilde{M}_{p_i,q_i}$ converge as varifolds to $\tau_{1,2}$, but by the area bounds \eqref{3bigger} \eqref{4bigger} the surfaces can never converge to any other Lawson surface $\tau_{n,m}$.  

For each $(n,m)$ there is an $\mathbb{S}^1$-action on $\mathbb{S}^3$ given by 
\begin{equation}
(z,w)\rightarrow (e^{in\theta}z, e^{im\theta}w).
\end{equation}
Let us refer to this as the $(n,m)$-action (or $\mathbb{S}^1_{n,m}$).  When $m=n=1$ the action is the usual Hopf action with orbit space equal to $\mathbb{S}^2$.  If $(n,m)\neq (1,1)$  the action is not free.  The isotopy group at the geodesic $\{z=0\}$ is $\mathbb{Z}_m$ and the isotropy group at $\{w=0\}$ is $\mathbb{Z}_n$.  If either $n=1$ or $m=1$ then only one of these geodesics has non-trivial isotropy group.  The quotient $\mathbb{S}^3/\mathbb{S}^1_{n,m}$ is a two-sphere with two (or one) orbifold points at the north and south pole.   The metric on the orbifold is invariant with respect to rotations about the $z$-axis. 

The Lawson $\tau_{n,m}$ surfaces are invariant under the $(n,m)$-action and are the only surfaces whose projection to $\mathbb{S}^3/\mathbb{S}^1_{n,m}$ contain the north or south pole in their supports.  

There is also a countable family of immersed tori invariant under the $(n,m)$-action which we denote $\{\tilde{A}_{n,m,a}\}_{a\in I}$. The surface $\tilde{A}_{n,m,a}$ is the lift of an immersed geodesic $A_{n,m,a}$ in the orbifold $\mathbb{S}^3/\mathbb{S}^1_{n,m}$.  This family is discussed further in the Appendix.  In particular,  we use the fact that all closed geodesics $\{A_{n,m,a}\}_{a\in I}$ have at least two points of self-intersection (Proposition \ref{intersecttwice}).

\subsection{Remaining limits}

Let us now consider limit of $\tilde{M}_{p,q}$ in the remaining cases. 
 
For each $0\leq r\leq \pi/2$ set 
\begin{equation}
C_r= \{(z,w)\in\mathbb{S}^3\; |\; |z|^2= \sin^2(r)\}
\end{equation}
so that
\begin{equation}
C_{\pi/4}=C.
\end{equation}
The family $\{C_r\}_{r\in [0,\pi/2]}$ is a cmc foliation of $\mathbb{S}^3$ away from two geodesics.   Note that each group $\mathbb{Z}^{q}_p$ preserves the parallel surfaces $C_t$ as sets.   We may cover each $C_r$ by
\begin{equation}
\phi_r:\mathbb{R}^2\rightarrow C_r\subset \mathbb{S}^3\subset\mathbb{R}^4,
\end{equation}
given by
\begin{equation}\label{coordinates}
\phi_r(x,y) = (\sin(r)e^{ix}, \cos(r)e^{iy}).
\end{equation}
Note that if we restrict $\phi_r$ to the fundamental domain $[0,2\pi]\times [0,2\pi]$, then $\phi_r$ gives a parameterization of $C_r$.   

Let $\{L(p_i, q_i)\}$ be a sequence of lens spaces with $p_i\rightarrow\infty$. For each $(x,y)\in C_r$, let $Orbit_i(x)$ denote the orbit of $(x,y)$ under the group $\mathbb{Z}_{p_i}^{q_i}$.  In other words

\begin{equation}
Orbit_i(x,y) = \phi_r(\{(x,y)+2\pi k(\frac{1}{p},\frac{q}{p})\in \mathbb{R}^2  \; | \; k=0,1,..,p_i-1\})
\end{equation}

An $(n,m)$-curve in the torus $C_r$ is the image of any line of slope $m/n$ in $\mathbb{R}^2$ to $C_r$ under $\phi_r$ (i.e. an $(n,m)$ torus knot).  

 Viana proved the following dichotomy (Lemma 3.3 in \cite{V}).  After passing to a subsequence one of the following holds:
\begin{enumerate}\label{possibility}
\item $Orbit_i(x)$ is getting dense in $C_r$ (i.e. for each $y\in C_r$ and $\varepsilon>0$, taking $i$ sufficiently large one can find a point in $Orbit_i(x)$ within $\varepsilon$ of $y$).
\item There exists integers $n,m$ and $k$ so that $Orbit_i(x)$ is contained in the union of $k$ parallel $(n,m)$-curves in $C_r$. 
\end{enumerate}
Let us refer to a sequence of lens spaces $\{L(p_i,q_i)\}_{i=1}^\infty$ as \emph{class 1} if they fall in (1) and \emph{class 2} with integers $(n,m,k)$ if they fall in (2). 

Give any positive integers  $k$, $n$ and $m$ with $n\leq m$ \footnote{Without loss of generality we assume $n\leq m$} we can find sequences $\{p_i\}_{i=1}^\infty$ and $\{q_i\}_{i=1}^\infty$ so that the lens spaces $\{L(p_i,q_i)\}$ fall  in class $2$ with positive integers $(n,m,k)$.    

To that end,  given integers $(n,m,k)$ let us write $m= nd+r$ with $0<r<n$ and $d\geq 1$,  and set $D=(d+1)$ and $B=k(n-r)$.   Then


\begin{lemma}[Realizing Configurations]
In each of the following cases,  for suitable sequences $p_i\rightarrow\infty$ we have:
\begin{enumerate}
\item If $n$ and $k$ are not both equal to $1$,  the lens spaces $L(knp_i+ B, p_i+D)$ are in class $2$ with integers $(n,m,k)$.
\item If $n=k=1$, then the lens spaces $L(p_i,m)$ are in class $2$ with integers $(1,m,1)$.    
\end{enumerate}
\end{lemma}
\begin{proof}
Item (2) is immediate so we prove (1).   The generator $\xi:=\xi_{knp_i+B, p_i+D}$ of the cyclic group corresponding to $L(knp_i+ B, p_i+D)$ maps (puling back to $\mathbb{R}^2$ under $\phi_r$):
\begin{equation}
\xi(0,0)=2\pi(\frac{1}{knp_i+B}, \frac{p_i+D}{knp_i+B}).
\end{equation}
Applying $\xi^{kn}$ we get
\begin{equation}\label{oyvey}
\xi^{kn} (0,0)=2\pi(\frac{kn}{knp_i+B}, \frac{kn(p_i+D)}{knp_i+B}).
\end{equation}
But using the definition of $D$ and $B$ we obtain 
\begin{align}\label{inn}
kn(p_i + D) &= knp_i+ knd + kn \\&= knp_i + km-kr+ kn\\&=knp_i +B + km.\label{fin}
\end{align}
Plugging \eqref{fin} in to \eqref{oyvey} we obtain (since we take only the fractional part)
\begin{equation}
\xi^{kn}(0,0) = 2\pi(\frac{kn}{knp_i+B}, \frac{km}{knp_i+B}).
\end{equation}
Thus the element $\xi^{nk}$ moves each point along a line of slope $m/n$.   

On the other hand,  the elements $\xi^0, \xi, \xi^2,...,\xi^{nk-1}$ move the origin $(0,0)$ to $nk$ points with increasing second coordinate less than $2\pi$.  One can see this because
\begin{equation}
\xi^{kn-1}(0,0) =2\pi(\frac{kn-1}{knp_i+B}, \frac{knp+B+km -p_i - d}{knp_i+B}) 
\end{equation}
and 
\begin{equation}
knp+B+km -p_i - d < knp_i+B \mbox{ for large } p_i.  
\end{equation}

For large $p_i$, the points $\{\xi^0, \xi, \xi^2,...,\xi^{nk-1}\}$ have (equally-spaced) $y$-coordinate approximately $2\pi\{0,  \frac{1}{kn}, \frac{2}{kn},...,\frac{kn-1}{kn}\}$.   Recall that a $(n,m)$ curve intersects the $y$-axis $n$ times.   Since the iterates of these points under $\xi^{kn}$ move along $(n,m)$ curves,  we obtain that all points lie on $k$ equally spaced parallel $(n,m)$ curves.  

Finally let us show that we can find an appropriate sequence $p_i\rightarrow \infty$.  Writing $s_i = p_i+D$ we want $s_i$ to be relatively prime to $As_i + E$, where $A=kn$ and $E=-knD+B$.   Then choose $s_i$ so that it has no common factors in its prime decomposition with $E$.
\end{proof}

For instance,  the lens spaces $L(2p+1, p+2)$ (for suitable $p\rightarrow\infty$) are in class $2$ with integers $m=3$, $n=2$, and $k=1$.   The lens spaces $L(5p,2p+2)$ fall into class $2$ with integers $k=5$, $m=2$ and $n=1$.  But the generic sequence of lens spaces are in class 1.

Now let us classify the remaining limits of various $\tilde{M}_{p_i,q_i}$:

\begin{thm}[Doubling of Clifford torus in general]\label{other}
Let $\{(p_i,q_i)\}_{i=1}^\infty$ be a sequence of pairs of positive integers with $1<q_i<p_i-1$, $\gcd(p_i,q_i)=1$, and where $p_i\rightarrow\infty$. Then after passing to a subsequence, we have one of the two possibilities for the sequence of lens spaces $L(p_i,q_i)$ and corresponding minimal surfaces $\tilde{M}_{p_i,q_i}$:
\begin{enumerate}
    \item The sequence of lens spaces is in class 1 and in the sense of varifolds  \begin{equation}
\lim_{i\rightarrow\infty} \tilde{M}_{p_i,q_i} =2C, 
\end{equation}
where the convergence is smooth at no point on $C$.
\item The sequence of lens spaces is in class 2 with $m=2$, and $k=1$ and in the sense of varifolds
\begin{equation}\label{odd}
\lim_{i\rightarrow\infty} \tilde{M}_{p_i,q_i}= \tau_{1,2}, 
\end{equation}
where $\tau_{1,2}$ is Lawson's \emph{immersed} Klein bottle.
\item The sequence of lens spaces is in class 2 with $m=1$ and $k=2$ and converge either to $2C$ smoothly away from two parallel $(1,1)$ curves (i.e.  great circles) or else to the union of two distinct Clifford tori.   
\item The sequence of lens spaces is in class 2 and not in any previous case and in the sense of varifolds
 \begin{equation}
\lim_{i\rightarrow\infty} \tilde{M}_{p_i,q_i} =2C,
\end{equation}
where the convergence is smooth away from $k$ equally spaced, parallel $(n,m)$-curves on $C$.
\end{enumerate}
\end{thm}


The case when the limit of the lens spaces is of class $2$ with $m=1$ and $k\geq 2$ was handled in Theorem \ref{findlimit} (corresponding to item (3) and part of (4) in Theorem \ref{other}).
\begin{proof}
Let us first assume that the lens spaces $L(p_i,q_i)$, up to a subsequence, fall into class $1$.  In other words, the orbit of any element $x$ on any cmc torus $C_r$ is getting dense as $i\rightarrow\infty$.  Since $\tilde{M}_{p_i,q_i}$ is connected, for each $i$  there exists an interval $A_i = (b_i,c_i)\subset (0,\pi/2)$ so that $\tilde{M}_{p_i,q_i}\subset \cup_{t\in A_i}C_t$ and moreover $\tilde{M}_{p_i,q_i}$ intersects every $C_t$ for $t\in A_i$. If $A_i$ does not shrink to $\pi/4$ as $i\rightarrow\infty$, then the limit of $\tilde{M}_{p_i,q_i}$ must contain every cmc torus $C_t$ for an interval of $t$ values,  which contradicts the area bound of $4\pi^2$.  Thus $b_i\rightarrow\pi/4$ and $c_i\rightarrow\pi/4$  and the sequence $\tilde{M}_{p_i,q_i}$ converges to the Clifford torus with some multiplicity. The multiplicity must be equal to $2$ by the area bounds \eqref{boundsforsurface} and Allard's theorem \cite{All}.  This completes the proof of (1).

Let us now consider a sequence of lens spaces of class $2$ with integers $(n,m,k)$.  The case $n=m=1$ was handled in Theorem \ref{findlimit},  so we assume without loss of generality that $m\geq 2$.   After passing to a subsequence, the corresponding minimal surfaces $\tilde{M}_{p_i,q_i}$ converge to a stationary integral varifold $\tilde{V}_{n,m,k}$ invariant under the $(n,m)$-action on $\mathbb{S}^3$.   This is because the iterates under $\mathbb{Z}_{p_i}^{q_i}$ of any point in the support of $\tilde{V}_{n,m,k}$ are contained on $k$ parallel $(n,m)$ torus knots, which become denser as $p_i\rightarrow\infty$.  After projecting  we obtain a stationary integral varifold $V_{n,m,k}$ in the corresponding orbifold two-sphere $\mathbb{S}^3/\mathbb{S}^1_{n,m}$ that is also invariant under $\mathbb{Z}_k$ rotations about the $z$-axis.  

As in the proof of Theorem \ref{findlimit},  let $\mathcal{S}$ denote the singular set of $V_{n,m,k}$,  and $\mathcal{S}_0=\mathcal{S}\cap\{N,S\}$ (where $N$ and $S$ denote the north and south pole, respectively, of $\mathbb{S}^3/\mathbb{S}^1_{n,m}$) and $\mathcal{S}_1=\mathcal{S}\setminus\mathcal{S}_0$.  As in Theorem \ref{findlimit}, let $G$ denote the equator in $\mathbb{S}^3/\mathbb{S}^1_{n,m}$.  We have the following possibilities:

\begin{enumerate}
\item $\mathcal{S}_0=\emptyset$ and $\mbox{supp}(V_{n,m,k})$ consists of a union of immersed geodesics with multiplicity $1$ or a single embedded geodesic with multiplicity at most $2$.  There are $k$ intersection points among the immersed geodesics if the support of $V_{n,m,k}$ is not smooth.   Furthermore,  the configuration is $\mathbb{Z}_k$-invariant.
\item $\mathcal{S}_0=\{N,S\}$ and $\mbox{supp}(V_{n,m,k})$ consists of $k$ equally spaced projections $\rho_{n,m}$ of Lawson surfaces $\tau_{n,m}$ and $\mathcal{S}_1=\emptyset$.  
\end{enumerate}


Assume first that $k\geq 3$.  By Proposition \ref{noteq} in the Appendix,  none of the curves $\{A_{n,m,a}\}_{a\in I}$ is $\mathbb{Z}_k$-equivariant.  Thus in case (1),  if $A_{n,m,a}$ is in the support of $V_{n,m,k}$ its $k$ iterates under $\mathbb{Z}_k$ must all be in the support, which gives rise to too much mass.  If $G$ is in the support then it must occur with multiplicity two by Allard's theorem \cite{All}.  Case (2) is impossible as $\rho_{n,m}$ is not $\mathbb{Z}_k$-invariant and two or more copies of $\rho_{n,m}$ give too much mass.  


If $k=2$,  for case (2),  the curve $\rho_{n,m}$ is $\mathbb{Z}_2$-equivariant but $\rho_{n,m}$ with multiplicity $1$ is not the limit of $\mathbb{Z}_2$-invariant surfaces.  In case (1),  from the considerations of the previous paragraph, the only possibility is that $V_{n,m,2} = 2G$.  

If $k=1$ then in case (1),  by Proposition \ref{intersecttwice}, each member of the countable family $\{A_{n,m,a}\}_{a\in I}$ has at least two points of self-intersection and thus cannot be in the support of $V_{n,m,1}$.   Thus in case (1) only $2G$ is possible.  In case (2),  when $m>2$,  by \eqref{3bigger} the area of $\tau_{n,m}$ is greater than $4\pi^2$ and so $\rho_{n,m}$ cannot be in the support of $V_{n,m,k}$.   Thus for $m>2$ case 2 does not occur and we indeed obtain $V_{n,m,2} = 2G$.  If $m=2$, however, then by \eqref{2less} we obtain \begin{equation}2\pi^2 < ||\tau_{1,2}||<4\pi^2\end{equation} and so we could have $V_{1,2,1} = \rho_{1,2}$ instead of $V_{1,2,1}= 2G$.  

One can likely exhibit an explicit two-parameter family for which $\tau_{1,2}$ is the optimal surface but it could be cumbersome to estimate the areas for the entire family.   Instead, we will show directly that $2G$ cannot arise as a limiting varifold by showing that the corresponding minimal surfaces $\tilde{M}_{p_i,2}$ cannot resemble, for large $p_i$, a doubling of a Clifford torus where the curvature is blowing up along a single $(1,2)$ curve on the torus $C$.

Suppose the sequence of minimal surfaces $\tilde{M}_{p_i,2}$ converges to $2C$, with non-smooth convergence over a single $(1,2)$ curve $y=2x$ in $C$. Let us denote this curve by $L$.

Fix a decreasing sequence $\epsilon_j\rightarrow 0$ and let $T_{\epsilon_j}(C)$ denote the open $\epsilon_j$-tubular neighborhood about $L$ in $C$. Let us define the tubular neighborhood in $\mathbb{S}^3$ (where $n$ is a choice of unit normal on $C$):
\begin{equation}
T_{\epsilon_j}(\mathbb{S}^3)=\{\exp_p(tn) \; |\; p\in T_{\epsilon_j}(C), t\in [-\pi/8,\pi/8]\}.
\end{equation}

Then $\tilde{M}_{p_i,2}\setminus T_{\epsilon_j}(\mathbb{S}^3)$ consists of two connected disks $C^i_1$ and $C^i_2$, which by Allard's theorem \cite{All}, for large $i$  each can be written as graphs $u^i_1(z)$, $u^i_2(z)$ over $C\setminus T_{\epsilon_j}(C)$.  Moreover, $u^i_2(z)<u^i_1(z)$ and $u^i_1(z), u^i_2(z)\rightarrow 0$ smoothly as $i\rightarrow\infty$ on $C\setminus T_{\epsilon_j}(C)$.  Let $p\in C\setminus T_{\epsilon_1}(C)$. Then for $i$ large enough, denote 

\begin{equation}
w_{ij}(z) = \frac{(u^i_1(z)-u^i_2(z))}{(u^i_1(p)-u^i_2(p))}.
\end{equation}

For each $j$, since $w_{ij}>0$ (as in the Appendix of \cite{CM}) by standard elliptic estimates, after passing to a subsequence we obtain from the sequence $\{w_{ij}\}_{i=1}^\infty$
a limiting function $w_j$ on the interior of $C\setminus T_{\epsilon_j}(C)$ so that \begin{equation}\label{iszero} w_j(p) =1.\end{equation}  The function $w_j$ is constant on all curves parallel to $L$ on $C$ and satisfies the Jacobi equation on its domain.  Furthermore, $w_j>0$ on the interior of $C\setminus T_{\epsilon_j}(C)$ by the Harnack inequality.

The functions $\{w_j\}_{j=1}^{\infty}$ are functions of one variable satisfying a homogeneous second order ODE.  Thus they must be among a two-dimensional family of solutions satisfying \eqref{iszero}.  

We may again pass to a subsequence of the family $\{w_j\}_{j=1}^\infty$ to obtain a smooth function $J$ on $C\setminus L$ that is not identically zero in light of \eqref{iszero}.\footnote{It would be natural to expect that $J$ extends to $C$ over $L$ as a continuous, piecewise smooth function.}  Pulling $J$ back to $\mathbb{R}^2$ via $\phi_{\pi/4}$ the function $J$ depends only on $y-2x$, and solves the Jacobi equation on the complement of the parallel lines $\{L_k\}_{k=-\infty}^{\infty}$, where
\begin{equation}
L_k = \{(x,y)\in\mathbb{R}^2\; |\; y - 2x = 2\pi k\}.
\end{equation}  
Moreover, $J$ restricted to $\mathbb{R}^2\setminus\bigcup_k L_k$ positive (though it might of course vanish on the $L_k$).\\

The Jacobi equation on the Clifford torus is given by \begin{equation}
L_C J(x,y)= 0
\end{equation} where, in the coordinates introduced in \eqref{coordinates}, we have \begin{equation}L_C=\Delta_C+4=2(\partial^2_x+\partial_y^2)+ 4=2\Delta_{\mathbb{R}^2} + 4.\end{equation}
\noindent
Thus on the domain of $J$ \begin{equation}\label{isdd}(\Delta_{\mathbb{R}^2}+2)J(x,y) = 0.\end{equation} 
\noindent
Since $J(x,y)$ only depends on $\xi = y-2x$ ($0\leq\xi<2\pi$) we obtain \begin{equation}J(x,y)= g(\xi)\end{equation} where $g$ is given by
\begin{equation}
g(\xi)= A\sin(\xi\sqrt{\frac{2}{5}}) + B\cos(\xi\sqrt{\frac{2}{5}}),
\end{equation}
for a suitable choice of $A$ and $B$. On the interval $\xi\in(0,2\pi)$, however, any such $g$ has at least one zero because
\begin{equation}
\sqrt{\frac{2}{5}} > \frac{1}{2}.
\end{equation}
Thus the purported $J$ cannot exist.   This completes the proof of item (2) that we obtain $\tau_{1,2}$ as a limiting varifold for the minimal surfaces $M_{p_i,2}$.  This completes the proof of Theorem \ref{other}.
\end{proof}

\begin{rmk}
If one applies the monotonicity formula to the cone over $V_{n,m,k}$ in $\mathbb{R}^4$ as in \eqref{easier} we obtain for any $x\in\mbox{supp}(V_{n,m,k})$ (where $\omega_3$ denotes the volume of the unit ball in $\mathbb{R}^3$)
\begin{equation}
\theta(\tilde{V}_{n,m,k},x)\leq \frac{||\tilde{V}_{n,m,k}||}{3\omega_3}=  \frac{||\tilde{V}_{n,m,k}||}{4\pi}\leq\pi.
\end{equation}
Thus the density of $V_{n,m,k}$ is at most three at any point, and one could alternatively try to classify such stationary integral varifolds to deduce that $V_k=2G$.   Unlike as in \eqref{easier} this does not seem to give sharp information.  
\end{rmk}

Finally we have,
\begin{prop}[Non-existence of genus 2 minimal surfaces]
For $p$ large enough, the lens space $L(p,1)$ does \emph{not} admit a genus $2$ minimal surface with area less than $4\pi^2/p$ (twice the area of the Clifford torus).   
\end{prop}

\begin{proof} 
Assume the contrary.  Then we have $p\rightarrow\infty$ and $\Sigma_p$ with genus $2$ and area less than $4\pi^2/p$.  Lifting to $\mathbb{S}^3$ we obtain minimal surfaces  $\tilde{\Sigma}_p$ of genus $g+1$ and area at most $4\pi^2$.  Consider a limiting stationary varifold $\tilde{\Sigma}_\infty$,  which is a union of Hopf fibers and its projection $V$ to $\mathbb{S}^2$ under the Hopf fibration.   The varifold $V$ has mass at most $4\pi$, and by the analysis of Theorem \ref{findlimit} it can have at most one singular point.   But there is no stationary varifold in $\mathbb{S}^2$ with one singular point, and thus $V$ must be an equator counted with multiplicity $2$.  As in Theorem \ref{findlimit}, the convergence of $\Sigma_p$ to the Clifford torus $C$ with multiplicity $2$ is smooth away from a single closed geodesic $V$ on $C$.   

Let us now show that such minimal surfaces do not exist for large $p$.   We give two proofs of this fact.

\emph{Method 1 (Nodal Domain):} Ros \cite{R} proved that any great sphere in $\mathbb{S}^3$ divides an embedded minimal surface into two connected components.  Choose a great sphere $S$ intersecting $C$ in a geodesic $V'$ parallel and close to $V$ on $C$.   Since $C\setminus V'$ is connected,  it follows that $S\cap C$ consists of other curve(s) $A$ disjoint from $V'$ (and $V$,  provided $V$ and $V'$ are chosen close enough).   Thus $C\setminus (V'\cup A)$ consists of at least one component $C'$ disjoint from $V$.   Since the convergence of $\Sigma_p$ to $2C$ is smooth away from $V$, it follows that $\Sigma_p\setminus (S\cap\Sigma_p)$ would have at least three components,  contradicting Ros' theorem. 

\emph{Method 2 (Impossible Jacobi field):}  
The argument resembles the proof of item (2) in Theorem \ref{other}.  Since the convergence of $\Sigma_p$ to $C$ is smooth with multiplicity $2$ away from $V$ we may consider the difference of the heights of the two graphical sheets comprising $\Sigma_p$ away from $V$. In this way we obtain a limiting function $J$ on $C\setminus V$ (c.f. the Appendix in Colding-Minicozzi \cite{CM}).  The function $J$ is smooth on its domain, and satisfies the Jacobi equation
\begin{equation}(\Delta_C+4)J = 0
\end{equation}
on $C\setminus V$.  Furthermore, by the Harnack inequality it follows that $J>0$ on $C\setminus V$. Finally, $J$ is constant on geodesics parallel to $V$.  

Pulling $J(x,y)$ back to $\mathbb{R}^2$ via $\phi_{\pi/4}$ we obtain a positive function defined on $\mathbb{R}^2\setminus \cup_k L_k$, where $L_k$ denotes the line $\{(x,y)\; |\;  y-x = 2\pi k\}\subset\mathbb{R}^2$.  Moreover, $J(x,y)$ satisfies the equation
\begin{equation}
2\Delta_{\mathbb{R}^2} J + 4 = 0 
\end{equation}
and depends on only on $y-x$. It is easy to see that 
\begin{equation}
J(x,y) = A\cos(y-x) + B\sin(y-x)
\end{equation}
for suitable $A$ and $B$.  Clearly any such function must be equal to zero along some line in $\mathbb{R}^2\setminus \cup_k L_k$.  This gives a contradiction and therefore the surfaces $\Sigma_p$ do not exist.
\end{proof} 

More generally, one might expect that $L(p,1)$ admits no genus $2$ minimal surface without any area assumption (or no index $2$ minimal surfaces).   The first three Almgren-Pitts widths introduced by Marques-Neves \cite{MN2} $\omega_1,\omega_2, \omega_3$  are realized by Clifford tori (since according to Proposition \ref{lens} there is an $\mathbb{RP}^2$ family of such minimal surfaces). The fourth width $\omega_4$ is likely realized by one of the Choe-Soret \cite{CS} surfaces.

\begin{rmk}
The first non-trivial eigenvalue for the Laplacian for a surface resembling $\Sigma_p$ converges to $1$ as $p\rightarrow\infty$.  Yau \cite{Y} has conjectured that $2$ is the lowest non-trivial eigenvalue for the Laplacian on embedded minimal surfaces in $\mathbb{S}^3$.
\end{rmk} 

\section{Distinct minimal surfaces in $\mathbb{S}^3$}

We first recall the following fact (Corollary 2.13 in \cite{T}):
\begin{lemma}\label{dothey}
The geodesics in $\mathbb{S}^3$ corresponding to $(a,b)\in\mathbb{S}^2\times\mathbb{S}^2$ and $(a',b')\in\mathbb{S}^2\times\mathbb{S}^2$ intersect if and only if $dist_{\mathbb{S}^2}(a,a') = dist_{\mathbb{S}^2}(b,b')$. 
\end{lemma}

For any geodesic $(a,b)\in\mathbb{S}^2\times\mathbb{S}^2$, there is a cmc Heegaard foliation of $\mathbb{S}^3$ of tori $\{F_{a,b}(t)\}_{t=0}^{\pi/2}$, where $F_{a,b}(t)$ denotes the cmc torus consisting of all points a distance $t$ from the geodesic $(a,b)$.  To specify the family,  suppose $b\in\mathbb{S}^2$.  Let $\{Y^b_t\}_{t=0}^{\pi/2}$ denote the family of round circles with $Y^b_0 = b$  and $Y^b_{\pi/2}= -b$.  Then in the notation of Section 4:
\begin{equation}\label{elf}
F_{a,b}(t) = \iota(a,Y^b_t).
\end{equation}
\noindent
Simillary,  we have the cmc torus given by 
\begin{equation}\label{elf2}
F'_{a,b}(t) = \iota(Y^a_t,  b).  
\end{equation}
\noindent
In fact,  as sets
\begin{equation}\label{equalyes}
F_{a,b}(t)=F'_{a,b}(t), 
\end{equation}
though the two surfaces are exhibited in \eqref{elf}, \eqref{elf2} as a union of different families of geodesics.

We can then show the following about how such tori intersect:
\begin{lemma}[Points of tangency between cmc tori]\label{isdisjoint}
Fix a geodesic $(a,b)\in\mathbb{S}^2\times\mathbb{S}^2$ and $\rho'\in (0,\pi/4]$.  There exists a neighborhood $\mathcal{N}\subset\mathbb{S}^2\times\mathbb{S}^2$ about $(a,b)$ so that for any geodesic $(a',b')\neq (a,b)$ contained in $\mathcal{N}$, and $\pi/4\geq \rho\geq\rho'$ if we denote by $t_0$ the supremum of all $t$ so that $F_{a',b'}(t)\subset F_{a,b}(\rho)$. Then $t_0$ is attained at some point in $(0,\rho)$ and
\begin{equation}
F_{a',b'}(t_0)\cap F_{a,b}(\rho)
\end{equation}
consists of 
\begin{enumerate}
    \item one closed geodesic in the case that $a=a'$ and $b\neq b'$ and also in the case that $a\neq a'$ and $b=b'$.
    \item two antipodal points if $a\neq a'$ and $b\neq b'$.
\end{enumerate}

\end{lemma}

\begin{proof}
Let $\mathcal{N}$ denote the subset of $\mathbb{S}^2\times\mathbb{S}^2$ corresponding to all great circles of $\mathbb{S}^3$ contained in the interior of the mean convex solid torus bounded by $F_{a,b}(\rho')$.

Let us assume without loss of generality that $a$ is the north pole and $b$ is the south pole.  If $a=a'$, then both cmc tori are a union of Hopf fibers, and correspond to lifts to $\mathbb{S}^3$ of (distinct) round circles on $\mathbb{S}^2$. Clearly, any tangency between two such circles (not consisting of the same circle) occurs at a single point in $\mathbb{S}^2$, which lifts via the Hopf fibration to a closed geodesic.   This gives the first case of (1).  

For item (2),  suppose $a\neq a'$ and $b\neq b'$.   Projecting the family $F_{a',b'}(t)$ to the second factor in $\mathbb{S}^2\times\mathbb{S}^2$ we obtain the foliation $Y^{b'}_t$ of round circle beginning at $b'$ and ending at $-b'$.  By Lemma \ref{dothey} this family will first hit $F_{a,b}(\rho)$ exactly at the time $t_0$ when  $Y^b_{t_0}$ contains a point a distance $dist_{\mathbb{S}^2}(a,a')$ away from the circle $Y^b_\rho$ on $\mathbb{S}^2$.   This occurs when $t_0=\rho-\mbox{dist}_{\mathbb{S}^2}(a,a')-\mbox{dist}_{\mathbb{S}^2}(b,b')$ and for that $t_0$ only one geodesic on $F_{a',b'}(t_0)$ intersects one geodesic on $F_{a,b}(\rho)$ (and all other geodesics are disjoint, since the circles $Y^b_\rho$ and $Y^{b'}_{t_0}$ are skew).  Since each cmc torus is foliated by geodesics, and any two geodesics that intersect do so in two points,  we obtain (2).  

Finally let us consider the second case in (1), where $a\neq a'$ and $b=b'$.   The first $t$ for which $F_{a',b'}(t)$ hits $F_{a,b}(\rho)$ happens at the same $t_0$ as in (2),  but since $b=b'$,  when it occurs,  $Y^{b'}_{t_0}$ and  $Y^{b}_{\rho}$ are parallel and not skew circles,  and so for each $p\in Y^{b}_{\rho}$ there exists a $p'\in Y^{b'}_{t_0}$ so that the geodesics $(a, p)$ and $(a',p')$ intersect twice.   The point $p'$ is obtained from $p$ by moving orthogonally downward from the circle $Y^{b}_{\rho}$ until hitting $Y^{b}_{t_0}$.    We claim that the intersection $F_{a',b'}(t_0)\cap F_{a,b}(\rho)$ consists of exactly one closed geodesic on the cmc torus $F_{a,b}(\rho)$.  To see this,  consider the point $a''$ in $\mathbb{S}^2$ obtained by moving along the equator in $\mathbb{S}^2$ containing $a$ and $a'$  from $a$, in the direction toward and past $a'$, a total distance $\rho$.    By \eqref{equalyes} the geodesic $(a'',b)$ is contained on both cmc tori.  Since $(a,p)$ and $(a',p')$ intersect twice for each $p\in Y^b_\rho$,  and the geodesic $(a'',b)$ already achieves this,  the intersection between the two cmc tori consists precisely of $(a'',b)$.

\end{proof}

Recall that $L(p,q)$ is isometric to $L(p,q')$ if and only if either \begin{equation}\label{aredifferent2} q+q'= 0\mod p \end{equation} or else \begin{equation}\label{aredifferent} qq' =\pm 1\mod p.\end{equation}  

Let us show 
\begin{thm}\label{arediff}
Suppose $p$ is sufficiently large.  Let $\{q_1,...,q_m\}$ be the set of all integers so that
\begin{equation}\{L(p,q_1),...,L(p,q_m)\}\end{equation} are pairwise non-isometric lens spaces (and $q_i\notin\{1,p-1\}$ for each $i$).   Then if $q_i\neq q_j$ it holds
\begin{equation}
\tilde{M}_{p,q_i}\neq\tilde{M}_{p,q_j}\mbox{ (up to isometry of }\mathbb{S}^3).
\end{equation}
\end{thm}
\begin{proof}
For $p$ sufficiently large, any surface $\tilde{M}_{p,q}$ (except $\tilde{M}_{p,2}$ and possibly a lens space of the form $\tilde{M}_{4k,2k\pm 1}$, where $k = p/4$) satisfies
\begin{equation}\label{damn}\tilde{M}_{p,q}\subset\{(z,w)\in\mathbb{S}^3\;|\; \frac{1}{8}\leq |z|^2\leq\frac{3}{8}\}.\end{equation}

 Indeed, suppose this is false.  Then we obtain a sequence of $L(p_i,q_i)$ and minimal surfaces $\tilde{M}_{p_i,q_i}$ that fail to satisfy \eqref{damn}. But since $\tilde{M}_{p_i,q_i}\rightarrow 2C$ by Theorem \ref{other} and Theorem \ref{findlimit} and the fact that varifold convergence implies Hausdorff convergence, this is a contradiction.   
 
 In the same way, we can choose $p$ large enough so that $\tilde{M}_{p,2}$ and $\tilde{M}_{4k,2k\pm 1}$ are isometric to none of the other minimal surfaces $\tilde{M}_{p,q}$ for any other $q$ nor are they isometric to each other.  Let us assume $p$ is so large so that both of these statements are true.  We thus remove $\tilde{M}_{p,2}$ and $\tilde{M}_{4k,2k\pm 1}$ from the lens spaces under consideration, and restrict to showing the others are distinct up to isometry. 

Let us also choose $p$ so large so that for each $q>1$, the $p$ iterates of $x\in\mathbb{S}^3\setminus (\{z=0\}\cup\{w=0\})$ under $\mathbb{Z}_p^q$ are not contained in a single geodesic.  This is possible by the dichotomy for lens spaces (\ref{possibility}) and the fact that $q>1$.\footnote{In $L(p,1)$ the orbit does move points along a single geodesic.} 

Finally we choose $p$ large enough so that any isometry $I\in O(4)$ with $I(\tilde{M}_{p,r})= \tilde{M}_{p,q}$ for some $q$ and $r$ has the property that the geodesic $I(\{z=0\})$ is contained in either $\{|z|^2\leq \frac{1}{8}\}$ or $\{|z|^2\geq \frac{3}{8}\}$.  Otherwise, we obtain sequences of lens spaces $L(p_i, q_i)$ and $L(p_i, r_i)$ where $I(\tilde{M}_{p,r_i})= \tilde{M}_{p,q_i}$ but $I(C)$ is a definite distance from $C$ in the $\mathcal{F}$-metric.  Since $\tilde{M}_{p,r_i}$ and $\tilde{M}_{p,q_i}$ both converge to $2C$ as $i\rightarrow\infty$, this is impossible.

Suppose the theorem were false.  Then without loss of generality suppose $q_1$ and $q_2$ satisfy $\tilde{M}_{p,q_1}=I(\tilde{M}_{p,q_2})$ for some isometry $I$ of $\mathbb{S}^3$ and $q_2<q_1\leq \lfloor p/2\rfloor$. 
We first claim that the isometry $I$ preserves the Clifford torus $C$ setwise.   Assume this is not the case. Then the image of $\{z=0\}$ under the isometry $I$ is a different geodesic, $Z_I=(a,b)\in\mathbb{S}^2\times\mathbb{S}^2$ with corresponding Clifford torus $C_I=F_{a,b}(\pi/4)$. By the choice of $p$, we can assume $Z_I$ is contained in a small enough neighborhood about the geodesic $\{z=0\}$ so that Lemma \ref{isdisjoint} applies with $\rho' = \pi/8$.  Let $\rho(p,q_1)>0$ be the minimal $z$ coordinate obtained on $\tilde{M}_{p,q_1}$ and let $y$ be a point on $F_{i,i}(\rho(p,q_1))=\{z=\rho(p,q_1)\}$ where it is attained.  By the choice of $p$, we have $\rho(p,q_1)>0$. Consider the cmc torus $L$ centered around $Z_I$ that is contained in $\{z\leq\rho(p,q_1)\}$ and tangent to $\{z=\rho(p,q_1)\}$ at some point $x$ (perhaps others to).  Apply a translation $t=(e^{i\theta_1}, e^{i\theta_2 })$ to $\mathbb{S}^3$ that takes $\tilde{M}_{p,q_1}$ to the surface $t\tilde{M}_{p,q_1}$,  where $t$ is chosen so that $ty = x$.

The surface $t\tilde{M}_{p,q_1}$ is invariant under the cyclic group $G_1 = \mathbb{Z}_p^{q_1}$ (since translations commute with elements of $G_1$).  It is also invariant under the conjugate group $J^{-1}G_2J$, where $J=  I^{-1}t^{-1}$.  One can see this because
\begin{align}
J^{-1}G_2J (t\tilde{M}_{p,q_1}) &= tIG_2I^{-1}t^{-1} (tI(\tilde{M}_{p,q_2}))\\ & = tIG_2(\tilde{M}_{p,q_2})\\&=tI(\tilde{M}_{p,q_2})\\&=t\tilde{M}_{p,q_1}.
\end{align}

The isometry $J$ takes the Heegaard foliation $\{F_{i,i}(t)\}_{t\in [0,\pi/4]}$ determined by cmc tori relative to $\{z=0\}$ to the Heegaard foliation of cmc tori relative to $Z_I$ given by $\{F_{a,b}(t)\}_{t\in [0,\pi/4]}$.  By construction, the cmc torus $L$ is contained in the solid torus $\{z\leq\rho(p,q_1)\}$ and there is a point of tangency between $L$ and $\{z=\rho(p,q_1)\}$ at $x$.  

By Lemma \ref{arediff} the tangency set between the cmc torus $L$ and $\{z=\rho(p,q_1)\}$ consists of two points, or a great circle. By the choice of $p$, the iterates of $x$ under the group $J^{-1}G_2 J$ are not contained on a single geodesic, nor do they consist of two points if $p>2$.  Therefore the $p$ iterates of $x$ under the cyclic group $J^{-1}G_2 J$ are not contained on $\{z=\rho(p,q_1)\}$ and among them are points with $z$-coordinate strictly less than $\rho(p,q_1)$, contradicting the minimality of the choice of $\rho(p,q_1)$.   This establishes the claim that the group $J^{-1}G_2J$ and also $I^{-1}G_2 I$ preserve $C$ setwise. 

On the other hand, we can enumerate all isometries $I\in O(4)$ that fix a Clifford torus.  Consider the following two involutive isometries of $\mathbb{S}^3$ that preserve the Clifford torus $C$:
\begin{equation}
\tau(z,w) = (w,z), 
\end{equation}
and the conjugation map
\begin{equation}
c(z,w) = (\overline{z},w). 
\end{equation}

There are eight isometries in $O(4)$ generated by $\{\tau,c\}$ that preserve $C$.  These make up the dihedral group $D_8$ (with $\tau c$ and its inverse having order $4$ and corresponding to a rotation by $\pi/4$).  Let $e$ denote the identity element of $D_8$: 
\begin{equation}\label{list}
D_8 = \{e, c, \tau, c\tau, \tau c,  c\tau c, \tau c\tau, c\tau c\tau \}.
\end{equation}

The isometry group of $C$ (as a manifold unto itself) is the semi-direct product of $D_8$ with the group of translations and thus any isometry of $C$ can be written as an element of $D_8$ followed by a translation.  Moreover, the action of an isometry $I\in O(4)$ on $\mathbb{S}^3$ is determined uniquely by its action on $C$.  To see this, suppose there are two distinct isometries, $I_1, I_2 \in O(4)$ that have the same action on $C$.  Then $I_2\circ I^{-1}_1$ fixes $C$ pointwise.  But the fixed point set of an isometry is totally geodesic, which the Clifford torus is not.  This implies that the fixed point set of $I_2\circ I^{-1}_1$ must be the entire manifold and thus $I_1=I_2$.  Since every isometry of $C$ can be exhibited by the restriction of an isometry in $O(4)$ it follows that the list \eqref{list} of isometries in $O(4)$ preserving $C$ is exhaustive up to composition with a translation\footnote{ Note that the antipodal map also preserves $C$ but is given by translation $(e^{\pi i}, e^{\pi i})$.}.

Therefore, since translations commute with each other, we can assume without loss of generality that the isometry $I$ is contained in the list \eqref{list}.  The surface $\tilde{M}_{p,q_1}$ is invariant under both $G_1$ and $I G_2 I^{-1}$.  

We next will show that for each possible choice of $I\in D_8$ the surface $M_{p,q_1}$ is invariant under a non-trivial finite group $\mathcal{O}$ of isometries that acts freely on the support of $M_{p,q_1}$.  Moreover, the group $\mathcal{O}$ acts by orientation-preserving isometries so that $M_{p,q_1}/\mathcal{O}$ is itself orientable.  On the other hand, a genus $2$ surface cannot be invariant under such a group since 
\begin{equation}
\chi(M_{p,q_1}) = -2 = |\mathcal{O}|\chi(M_{p,q_1}/\mathcal{O}), 
\end{equation}
or
\begin{equation}
\mbox{genus}(M_{p,q_1}/\mathcal{O}) = 1+\frac{1}{|\mathcal{O}|},
\end{equation}
which implies that the genus of $M_{p,q_1}/\mathcal{O}$ could not be an integer.  Thus this will complete the proof.

Let us first assume $I =e$.  Since $L(p,q_1)$ and $L(p,q_2)$ are not diffeomorphic, we can assume without loss of generality $q_1-q_2\leq \lfloor p/2\rfloor$. It follows that we have \begin{equation}\label{istrivialno}q_1-q_2\neq 0\mod p.\end{equation}

Since the minimal surface $\tilde{M}_{p,q_1}=\tilde{M}_{p,q_2}$ is invariant under both $\mathbb{Z}_p^{q_1}$ and $\mathbb{Z}_p^{q_2}$,  we get
\begin{equation}
\xi^{p-1}_{p,q_2}.\xi_{p,q_1}(\tilde{M}_{p,q_1})= \tilde{M}_{p,q_1}, 
\end{equation}

Fixing an $x=(a,b)$ we get
\begin{align}
\xi^{p-1}_{p,q_2}.\xi_{p,q_1}(a,b) &= (e^{2\pi i(p-1+1)/p}a,e^{2\pi i((p-1)q_2+q_1)/p}b)\\ &=(a,e^{(q_1-q_2)/p}b). 
\end{align}
The point $(a,b)$ is rotated only in the second factor and by \eqref{istrivialno} this rotation is non-trivial.  The subgroup $\mathcal{O}\subset O(4)$ generated by $\mathbb{Z}_p^{q_1}$ and $\mathbb{Z}_p^{q_2}$ that rotates only the second factor preserves a fundamental domain for both $L(p,q_1)$ and $L(p,q_2)$.  The action of $\mathcal{O}$ is not free and fixes the circle $\{w=0\}$.  Because $\tilde{M}_{p,q_1}$ is disjoint from the circle $\{w=0\}$, it follows that $\mathcal{O}$ acts freely on $\tilde{M}_{p,q_1}$ .


Secondly assume $I=\tau$.  Observe that 
\begin{equation}
\tau\xi_{p,q_2}\tau^{-1}(a,b) = (ae^{2\pi i q_2/p}, be^{2\pi i/p}).  
\end{equation}

Then also applying a power of the generator of $G_1$:
\begin{equation}
\xi_{p,q_1}^{(p-q_2)}\tau\eta_{p,q_2}\tau^{-1}(a,b) = (a, b e^{2\pi i(1+q_1(p-q_2))/p}).
\end{equation}

But \begin{equation}1+q_1(p-q_2)=1-q_2 q_1 \neq 0\mod p\end{equation} because by assumption $L(p,q_1)$ and $L(p,q_2)$ are not isometric. \eqref{aredifferent}. 

Thirdly assume $I=c$.  Then 

\begin{equation}
c\xi_{p,q_2} c^{-1}(a,b) = (ae^{-2\pi i/p}, be^{2\pi i q_2/p}), 
\end{equation}

Applying $\xi_{p,q_1}$ to the result we get
\begin{equation}
\xi_{p,q_1} c\xi_{p,q_1} c^{-1}(a,b) = (a, be^{2\pi i (q_2+q_1)/p}). 
\end{equation}
\noindent
Because $q_1+q_2\neq 0\mod p$ we complete this case as well.    

Fourthly, assume $I = \tau c$.  Then we obtain
\begin{equation}
(\tau c)\xi_{p,q_2} (\tau c)^{-1} (a,b) = (ae^{2\pi iq_2 /p}, be^{-2\pi i/p}), 
\end{equation}
\noindent
Applying then the appropriate power of $\xi_{p,q_1}$ we get 
\begin{equation}
\xi_{p,q_1}^{p-q_2} (\tau c)\xi_{p,q_2} (\tau c)^{-1} (a,b) = (a, be^{2\pi i(q_1(p-q_2)-1)/p}).
\end{equation}
But as in the second case \begin{equation} q_1(p-q_1)-1 = -q_1q_2-1\neq 0 \mod p\end{equation} by \eqref{aredifferent}.  Thus we have a non-trivial group of rotations in this case as well preserving and acting freely on $M_{p,q_1}$.

The remaining four cases are similar, and we merely list the appropriate group elements to be applied which give a rotation of the second factor. 

For $I = c\tau$, we consider the rotation:
\begin{equation}
\xi_{p,q_1}^{q_2} (c\tau)\xi_{p,q_2}(\tau c)  (a,b) = (a, be^{2\pi i(q_1 q_2 +1)/p})\neq (a,b).
\end{equation}

For $I = c\tau c$, we consider the rotation:
\begin{equation}
\xi_{p,q_1}^{q_2} (c\tau c)\xi_{p,q_2} (c\tau c)^{-1} (a,b) = (a, be^{2\pi i(q_1 q_2 -1)/p})\neq (a,b).
\end{equation}

For $I = \tau c \tau $, we consider the rotation:
\begin{equation}
\xi_{p,q_1}^{p-1} (\tau c\tau)\xi_{p,q_2} (\tau c \tau)^{-1} (a,b) = (a, be^{-2\pi i(q_1+q_2)/p})\neq (a,b).
\end{equation}

Finally for $I = c\tau c \tau $, we consider the element:
\begin{equation}
\xi_{p,q_1} (c\tau c\tau )\xi_{p,q_2} (\tau c\tau c) (a,b) = (a, be^{2\pi i(q_1 -q_2 )/p})\neq (a,b).
\end{equation}

This completes the proof.


\end{proof}
 The number of distinct lens spaces with fundamental group $\mathbb{Z}_p$ tends to infinity as $p$ tends to infinity:
\begin{lemma}\label{burnside}
The number of lens space (up to isometry) with fundamental group equal to $\mathbb{Z}_p$ is at least $\frac{\phi(p)}{4}$,  where $\phi(p)$ denotes the Euler totient function. \footnote{Recall that the Euler totient function $\phi(p)$ is equal to the number of positive integers less than and relatively prime to $p$. }    

\end{lemma}
\begin{proof}
Let  $G_p$ denote the group of order $\phi(p)$ consisting of units in $\mathbb{Z}_p$,  i.e.,  positive integers less than $p$ that are relatively prime to $p$.   Let $\mathbb{Z}_2\times\mathbb{Z}_2$ denote the group acting on $G_p$ in the first factor by additive inverse,  and the second by multiplicative inverse.   Note that additive inverse preserves the set of units.   By \eqref{aredifferent} and \eqref{aredifferent2} the number of distinct lens spaces up to isometry is equal to the number of orbits under this action.    If $\mathbb{Z}_2\times\mathbb{Z}_2$ acted freely on $G_p$,  there would be $\phi(p)/4$ orbits.  If the action is not free,  the number of orbits can only increase.   Indeed  by Burnside's lemma,  the number of orbits of $\mathbb{Z}_2\times\mathbb{Z}_2$ acting on $G_p$ is given by 
\begin{equation}
N(p)= \frac{1}{4}\sum_{g\in \mathbb{Z}_2\times\mathbb{Z}_2} |G_p^g|, 
\end{equation}
where $|G_p^g|$ denotes the number of elements in $G_p$ fixed by the group element $g$.   For the identity element $e$, 
\begin{equation}
|G_p^e| = \phi(p).  
\end{equation}
Thus we obtain 
\begin{equation}
N(p)\geq \phi(p)/4.
\end{equation}
\end{proof}

Since $\phi(p)\geq C \frac{p}{\log\log(p)}$ for $p$ large enough,  by putting together Theorem \ref{arediff} with Lemma \ref{burnside} we obtain finally:

\begin{thm}
There holds
\begin{equation}
\lim_{g\rightarrow\infty} |\mathcal{S}_g|=\infty.  
\end{equation}
\end{thm}
\noindent

\section{Appendix: Hsiang-Lawson tori}
For any $a\in (0,\pi/2)$ and positive integers $n< m$ with $\gcd(n,m) = 1$ define the \emph{period}
\begin{equation}\label{period}
P_{n,m,a} = \frac{2\sin{a}}{m} \int_a^{\pi-a} \sqrt{\frac{n^2\cos^2{(x/2)}+ m^2\sin^2{(x/2)}}{\sin^2{x}-\sin^2{a}}} \frac{dx}{\sin{x}}.
\end{equation}
Following Hsiang-Lawson \cite{HL}, for any $a$ such that $P_{n,m,a}$ a rational multiple of $\pi$,  we obtain an immersed torus in $\mathbb{S}^3$ unvariant under the $(n,m)$-action on $\mathbb{S}^3$ described in Section \ref{lawsonsection}. 

Period functions were extensively studied for the Otsuki action (\cite{O}, \cite{O2}) governing rotationally symmetric tori in $\mathbb{S}^3$, and were used by Andrews and Li \cite{AL} to complete the classification of constant mean curvature embedded tori in $\mathbb{S}^3$.  To the author's knowledge,  the period function for the $(n,m)$-actions has not been studied.  We need the following monotonicity and limiting property (cf. Proposition 13 in \cite{AL}):

\begin{thm}[Mononoticity of Period] \label{monotonicity} 
For each pair of positive integers $(n,m)$ with $n< m$ and $\gcd(n,m)=1$ we have:
\begin{enumerate} 
\item $P_{n,m,a}$ is strictly increasing in $a$ for $a\in (0,\pi/2)$.
\item There holds
\begin{equation}\label{limitasyougo}
\lim_{\alpha\rightarrow\pi/2} P_{n,m,a} = 2\pi \sqrt{\frac{n^2+m^2}{2m^2}}.
\end{equation}
\end{enumerate}
Thus we obtain
\begin{equation}
 P_{n,m,a}<  2\pi \sqrt{\frac{n^2+m^2}{2m^2}}<2\pi \mbox{  for all  } a\in(0,\pi/2).
\end{equation}
\end{thm}

The integrand in \eqref{period} diverges at its boundary points and furthermore the derivative of the integrand in $a$ is not integrable.  This makes differentiating \eqref{period} with respect to $a$ delicate.  Otsuki \cite{O2} used complex analysis to express similar integrals as contour integrals which are easier to differentiate.

\begin{proof}
Let us first show \eqref{limitasyougo}.    We have
\begin{equation}\label{tops}
\lim_{x\rightarrow\pi/2} \sqrt{n^2\cos^2{(x/2)}+ m^2\sin^2{(x/2)}} = \sqrt{\frac{n^2+m^2}{2}}.
\end{equation}
On the other hand
\begin{equation}\label{bottom}
\int \frac{1}{\sin{x}}\frac{dx}{\sqrt{\sin^2{x}- \sin^2{a}}}  = -\tan^{-1}(\frac{\sqrt{2}\cos{x}\sin{a}}{1-2\sin^2{\alpha}-\cos{2x}}).
\end{equation}
Combining \eqref{tops} and \eqref{bottom} we obtain (2).   

To prove (1), we first change variables by $y=\cos(x)$ in \eqref{period} to obtain
\begin{equation}
P_{n,m,a}= \frac{2\sin{a}}{m} \int_{-\cos{a}}^{\cos{a}} \frac{1}{1-y^2}\frac{\sqrt{A-By}}{\sqrt{\cos^2{a}-y^2}}dy, 
\end{equation}
where $A= \frac{n^2+m^2}{2}$ and $B=\frac{m^2-n^2}{2}$.

Consider the function:
\begin{equation}
f(z,a) = \frac{2}{m}\frac{\sin{a}}{1-z^2}\frac{\sqrt{A-Bz}}{\sqrt{\cos^2{a}-z^2}}.
\end{equation}
A single-valued branch for $f(z,a)$ may be chosen away from the branch cut 
\begin{equation}
B_1=\{z=x+iy\in\mathbb{C}\; |\; -\cos{a}\leq x\leq -\cos{a}\}
\end{equation} 
as well as the branch cut $B_2$ given by
\begin{equation}
B_2=\{z=x+iy\in\mathbb{C}\; |\; x\geq \frac{A}{B}\}.
\end{equation} 
\noindent
Note that $\frac{A}{B}= \frac{n^2+m^2}{m^2-n^2}>1$ so that the branch cuts $B_1$ and $B_2$ are disjoint.  Let $\gamma_1$ be a closed curve oriented clockwise enclosing $B_1$ but not enclosing either $B_2$ or $\pm 1$.  

If we consider the limit of $f(z,a)$ as $z$ approaches $B_1$ from above, we get the negative of the value we get from approaching from below.  Thus by Cauchy's theorem we obtain
\begin{equation}\label{firstintegral}
P_{n,m,a}= \frac{1}{2} \int_{\gamma_1} f(z,a) dz.  
\end{equation}

Let $H_{h}$ denote the $h$-tubular neighborhood in $\mathbb{C}$ about the half-line $\{z\geq\frac{A}{B}\}$ (on the real line).  Let $D_R$ denote the solid disk of radius $R$ about the origin. Then let $W_{R,h}=D_R\setminus H_h$ and set $C_{R,h} = \partial W_{R,h}$.  Orient the curve $C_{R,h}$ clockwise.

We may change the curve of integration in \eqref{firstintegral} to $C_{R,h}$ at the expense of picking up residues of $f(z,a)$ at $z=\pm 1$.  The residues, however, are constants independent of $a$:
\begin{equation}\label{finalperiod}
P_{n,m,a} = \frac{-\pi n}{m} + \pi+ \frac{1}{2}\int_{C_{R,h}} f(z,a) dz. 
\end{equation}
We may then differentiate \eqref{finalperiod} with respect to $a$ and obtain
\begin{equation}\label{final}
P'_{n,m,a}= \frac{1}{2}\int_{C_{R,h}} \frac{\partial f(z,a)}{\partial a} dz.
\end{equation}

Thus we obtain:

\begin{equation}\label{alongreal}
P'_{n,m,a} = \frac{\cos{a}}{m}\int_{C_{R,h}}  \frac{\sqrt{A-Bz}}{(\cos^2{a}-z^2)^{3/2}}dz
\end{equation}
Finally we may take $R\rightarrow\infty$ and $h\rightarrow 0$ so that (as one may check) the integral in \eqref{alongreal} reduces to two integrals along the branch cut $B_2$.  Both of these integrals have the same numerical value (even though they have opposite orientations and are taken along the same contour, they add instead of cancel because of the jump corresponding to multiplying by $-1$ across the branch cut).  

Thus we obtain (pulling out an $i$ from the numerator of \eqref{alongreal} and $-i$ from the denominator):

\begin{equation}\label{derivative}
P'_{n,m,a} = \frac{2\cos{a}}{m}\int_{A/B}^\infty\frac{\sqrt{Bx-A}}{(x^2-\cos^2{a})^{3/2}}dx
\end{equation}
Note that the integral \eqref{derivative} is convergent since the integrand is of order $x^{-5/2}$ near infinity.  

Since $P'_{n,m,a}>0$ from the form of \eqref{derivative} for any $a\in(0,\pi/2)$, this completes the proof of the theorem.  
\end{proof}

Following Hsiang-Lawson (page 32 in \cite{HL}), whenever $\frac{P_{n,m,a}}{2\pi}$ is a rational number, one obtains an immersed minimal torus in $\mathbb{S}^3$ that is invariant under the $(n,m)$-action on $\mathbb{S}^3$.   

More specifically, for any $a\in (0,\pi/2)$ there exists a graph \begin{equation}\phi_a(\theta):\mathbb{R}\rightarrow\mathbb{R}\end{equation} with period $P_{n,m,a}$.  When $\frac{P_{n,m,a}}{2\pi}$ is a rational number $r/s$,  the graph $\phi_a(\theta)$ closes up on the interval $[0, 2 \pi t]$ for some (smallest) natural number $t$ depending on $r$ and $s$ and projects modulo $2\pi$ to give an \emph{immersed} curve $A_{n,m,a}$ parameterized by a circle $[0,2\pi]$.  This circle lifted to $\mathbb{S}^3$ corresponds to the immersed torus $\tilde{A}_{n,m,a}$ in $\mathbb{S}^3$. 

In fact, 

\begin{equation}\label{periodbounds}
\frac{1}{2}<\frac{P_{n,m,a}}{2\pi} < \sqrt{\frac{n^2+m^2}{2m^2}}<1.
\end{equation}

Theorem \ref{monotonicity} gives the upper bound in \eqref{periodbounds}.  The lower bound is immediate because by the discussion above, if it failed, one would obtain an embedded torus in $\mathbb{S}^3$ distinct from the Clifford torus.  This violates \cite{B}.  We obtain:

\begin{prop}\label{intersecttwice}
Each curve ${A}_{n,m,a}$ has at least two points of self-intersection.  
\end{prop}
\begin{proof} 
When $\frac{P_{n,m,a}}{2\pi}$ is a rational number $r/s$,  the graph $\phi_a(\theta)$ closes up on the interval $[0, 2 \pi t]$ for some (smallest) natural number $t$ depending on $r$ and $s$ and projects modulo $2\pi$ to give an immersed curve parameterized by a circle.  By the first inequality in \eqref{periodbounds} we have $c\geq 2$.  

Consider the graph $\phi_a(\theta)$ restricted to the intervals $[0,P_{n,m,a}]$ and $[2\pi, 2\pi + P_{n,m,a}]$.   These intervals correspond to the same interval when projected to the circle $[0,2\pi]$ because $P_{n,m,a}< 1$ by the last inequality in \eqref{periodbounds}.  But two translations of a $P_{n,m,a}$-periodic curve have to intersect at least twice on any interval of length $P_{n,m,a}$.  Thus there are at least two self-intersection points for any curve $A_{n,m,a}$.
\end{proof}

We also have \begin{prop}\label{noteq}
No curve $A_{n,m,a}$ is $\mathbb{Z}_k$-invariant for $k\geq 2$.
\end{prop}
\begin{proof}
If $\{A_{n,m,a}\}_{a\in I}$ were $\mathbb{Z}_k$-invariant it would arise from the graph of a periodic function $\phi_a(\theta)$  with period $P_{n,m,a}\leq \frac{2\pi}{k}\leq \frac{2\pi}{2} $.  But by the lower bound in \eqref{periodbounds}, in fact $\frac{P_{n,m,a}}{2\pi} >\frac{1}{2}$.  
\end{proof}
\printbibliography

\end{document}